\documentclass[12pt,a4paper,reqno]{amsart}
\pdfoutput=1
\usepackage[hidelinks,pdfstartview=FitH]{hyperref} 
\usepackage{amssymb,tikz}
\usepackage{mathtools}
\usepackage[all]{xy}
\usepackage[utf8]{inputenc}
\usepackage{ogonek} 
\usepackage[top=3cm,right=3cm,left=3cm,bottom=3cm,marginparsep=15pt,marginparwidth=2cm,footskip=20pt]{geometry}

\allowdisplaybreaks[4]

\newtheorem{prop}{Proposition}[section]
\newtheorem{df}[prop]{Definition}
\newtheorem{lemma}[prop]{Lemma}
\newtheorem{thm}[prop]{Theorem}
\newtheorem{cor}[prop]{Corollary}
\newtheorem{conj}[prop]{Conjecture}
\newtheorem*{thm*}{Theorem}
\theoremstyle{definition}
\newtheorem{rem}[prop]{Remark}
\newtheorem{ex}[prop]{Example}
\newtheorem{notation}[prop]{Notation}
\numberwithin{equation}{section}

\title[The local-triviality dimension of actions of compact quantum groups]{The local-triviality dimension of actions\\ of compact quantum groups} 

\author[E.~Gardella]{Eusebio Gardella} 
\address[E.~Gardella]{Mathematisches Institut, Fachbereich Mathematik und Informatik
der Universitat M\"unster, Einsteinstrasse 62, 48149 M\"unster, Germany}
\email{gardella@uni-muenster.de}

\author[P. M.~Hajac]{Piotr M.~Hajac}
\address[P. M.~Hajac]{Instytut Matematyczny, Polska Akademia Nauk, ul. \'Sniadeckich 8, Warszawa, 00-656 Poland}
\email{pmh@impan.pl }

\author[M.~Tobolski]{Mariusz Tobolski}
\address[M.~Tobolski]{Instytut Matematyczny, Polska Akademia Nauk, ul. \'Sniadeckich 8, Warszawa, 00-656 Poland}
\email{mtobolski@impan.pl }

\author[J.~Wu]{Jianchao Wu}
\address[J.~Wu]{Department of Mathematics, Penn State University, 109 McAllister Building, University Park, PA 16802 USA}
\email{jianchao.wu@psu.edu}

\subjclass[]{}


\begin{document}

\begin{abstract}
We define the local-triviality dimension for actions of compact quantum groups on unital C*-algebras. The resulting 
compact quantum principal bundle is said to be locally trivial when this dimension is finite. For commutative \mbox{C*-al}\-ge\-bras, 
this notion 
recovers the standard definition of local triviality of compact principal 
bundles. We  prove that  actions with finite local-triviality dimension are automatically free. Then we apply this new 
notion to prove the noncommutative Borsuk--Ulam-type conjecture under the assumption that a compact quantum group admits 
a non-trivial classical subgroup whose induced action has finite local-triviality dimension. This is a noncommutative extension of the Borsuk--Ulam-type theorem for locally trivial principal bundles. 
\end{abstract}
\maketitle
{\small\tableofcontents}
\section{Introduction}\noindent

The noncommutative geometry program of Connes \cite{Co94}, taking inspiration from a variety of areas including 
representation theory, quantum mechanics, and index theory, aims at building new mathematical tools by providing 
noncommutative (or ``quantum'') generalizations of classical mathematical theories. One starting point of this program is the 
Gelfand--Naimark natural duality between commutative \mbox{C*-al}gebras
and locally compact Hausdorff spaces~\cite{gn-43}, which leads to the perception of the study of general, possibly 
noncommutative C*-algebras as 
the theory of \emph{noncommutative topology}. 
Many other important mathematical subjects have since found their noncommutative analogues, e.g., measure theory,
spin geometry, and topological $K$-theory. The resulting new tools have made far-reaching impact in topology, 
representation theory, ergodic theory, mathematical physics, etc. 


In this paper, we are interested in the concept of a principal bundle which is pivotal in algebraic topology and
provides a rigorous mathematical description of gauge field theories in physics. For a compact group $G$, a compact 
Hausdorff principal $G$-bundle in the sense of H.~Cartan may be taken as a compact Hausdorff space with a free continuous 
$G$-action. However, it is often desirable in applications to impose the condition of \emph{local triviality}, 
namely that each point has a neighborhood of the form $G \times W$, where $W$ is some
open subset of the orbit space and $G$ acts only on the first coordinate by translation. Such a neighborhood is called 
\emph{trivializable}, as it represents a trivial bundle. Local triviality is incorporated into the definition of principal bundles per 
Steenrod, but for the sake of clarity, we will stick to Cartan's formulation in this paper and explicitly state local triviality when 
needed. 

Attempts to extend locally trivial principal bundles to the realm of noncommutative topology go back already to the 1990's~\cite{bk-94,p-94}.
Unlike  the notion of a vector bundle, which immediately entered noncommutative topology as a finitely generated projective 
module via the Serre--Swan theorem, the notion of a \emph{locally trivial} compact Hausdorff principal bundle resisted 
generalization to 
noncommutative topology, largely due to the latter's \emph{global} nature.

Already  the concept of a free action on C*-algebras proved to be difficult to formulate in a satisfactory 
manner~\cite{BCH17}.
A key problem in imposing the local-triviality condition in this setting is the lack of a C*-al\-ge\-braic formulation of the notion
of an \emph{open} cover. Therefore, there came first 
the definition of a piecewise trivial compact quantum principal bundle, using
appropriate families of ideals to define a noncommutative finite \emph{closed} cover~\cite{HKMZ11,hkz-12}.
However, as shown in \cite{BHMS07}, 
a piecewise trivial compact Hausdorff principal bundle need not be locally trivial, so the problem
of introducing local triviality to noncommutative geometry remained unsolved.

In this paper, we introduce a notion of locally trivial compact quantum principal bundles. 
Our local triviality is characterized by the finiteness of a dimension concept which we call the \emph{local-triviality dimension}, defined for actions of compact quantum groups on unital C*-algebras. 
Our approach is inspired by the theory of the Rokhlin dimension \cite{HWZ15} used in and around the classification program of unital simple 
separable nuclear C*-algebras and is also motivated by the noncommutative Borsuk--Ulam-type conjecture \cite{BDH15}. 

To explain 
how we circumvent the need for open covers in our definition, we first describe an equivalent characterization of local triviality in the classical setting. We begin by noting that a prominent example of locally trivial compact Hausdorff principal bundles for a compact group $G$ is given by applying Milnor's join construction to $G$. More precisely, the \emph{join} $G \ast G$ is defined to be the quotient topological space of $G \times G \times [0,1]$ by collapsing each of the two copies of $G$ at the endpoints of $[0,1]$, one at a time. When equipped with the diagonal translation action, the join $G \ast G$ becomes a locally trivial compact Hausdorff principal $G$-bundle. We may iterate this construction to obtain the \emph{multi-joins} 
\[
	E_n {G} :=\underbrace{G \ast G \ast \ldots \ast G}_{n+1}
\]
for any natural number $n$. For example, when $G$ is the two-element group, the multi-join $E_n {G}$ is identified with the $n$-sphere with the antipodal action. These multi-joins 
provide the desired equivalent characterization of local triviality: 
a compact Hausdorff principal bundle is locally trivial if and only if it admits an equivariant continuous map into such a multi-join. This fact is proved by a partition-of-unity argument. 

Observe that this equivalent characterization is \emph{global} in nature, in the sense that it avoids the need to talk about open neighborhoods. Therefore, it serves as a perfect point of departure in our quest to generalize local triviality to the noncommutative setting. In addition, it turns out that the smallest $n$ needed for the existence of an equivariant continuous map into $E_n {G}$ is exactly one less than the smallest number of trivializable open sets needed to cover the compact Hausdorff principal bundle. 
As we shall see below, it is particularly meaningful to keep track of this number $n$ as a measurement of the complexity of the principal bundle. This is our local-triviality dimension in the classical setting\footnote{The terms \emph{$G$-index} and \emph{Schwarz genus} have also been used for this number in the literature. }. Thus in summary, the local triviality of a compact Hausdorff principal bundle is characterized by the finiteness of its local-triviality dimension. 

To generalize this dimension concept to the noncommutative setting, we just need to invoke the Gelfand--Naimark duality: the local-triviality dimension of an action by a compact quantum group $\mathbb{G}$ on a unital C*-algebra $A$ is the smallest number $n$ such that there exists an equivariant $*$-homomorphism from the C*-algebra of the ``$n$-th multi-join'' of $\mathbb{G}$ into $A$. There are, however, some subtleties in constructing multi-joins for compact quantum groups. A naive approach is to dualize the construction in the classical setting and define the C*-algebra of the join of $\mathbb{G}$ with itself as a C*-subalgebra of $C(\mathbb{G}) \otimes C(\mathbb{G}) \otimes C([0,1])$, where $\otimes$ denotes the minimal or maximal tensor product of C*-algebras. Unfortunately, unlike classical groups, there is in general no well-defined diagonal action by $\mathbb{G}$ on $C(\mathbb{G}) \otimes C(\mathbb{G})$. 
There are two ways to fix this issue: 
\begin{enumerate}
	\item Replacing the tensor product $C(\mathbb{G}) \otimes C(\mathbb{G})$ by an \emph{amalgamated free product} $C(\mathbb{G}) \ast_{\mathbb{C}} C(\mathbb{G})$. This results in a construction that we call the \emph{free noncommutative join}, denoted by $C(\mathbb{G}) {\mathrlap{+}\times} C(\mathbb{G})$ and, for the iterated version, $C(E_n^{{\mathrlap{+}\times}}\mathbb{G})$. Using this as the C*-algebra of the ``$n$-th multi-join'' in the above statement leads to our definition of the \emph{local-triviality dimension}. 
	\item Replacing the tensor product $C(\mathbb{G}) \otimes C(\mathbb{G})$ by a \emph{braided tensor product}, which admits a diagonal action. Equivalently, we may ``twist'' the above ``naive join'' by altering one of the two endpoints on $[0,1]$ in a suitable way and replace the diagonal action by the action on the second tensor factor alone, which is well defined. Either way, the resulting C*-algebra is called the \emph{equivariant noncommutative join}, denoted by $C(\mathbb{G}) \overset{\Delta}{\circledast} C(\mathbb{G})$ and, for the iterated version, $C(E_n^{\Delta}\mathbb{G})$, where $\Delta$ stands for the comultiplication in $C(\mathbb{G})$. Using this in place of the free noncommutative join, we arrive at the definition of the \emph{strong local-triviality dimension}, so named because it assumes greater values than the local-triviality dimension in general. 
\end{enumerate}
Both approaches generalize the local-triviality dimension in the classical setting, though they differ in general, even when the acting group is classical. An example is given by the antipodal action on the free spheres, which may be identified with the free noncommutative multi-joins $C(E_n^{{\mathrlap{+}\times}}{G})$ for $G$ being the two-element group. Compared to the strong local-triviality dimension, there seem to be fewer tools for giving lower bounds of the local-triviality dimension, as it turns out all the free noncommutative multi-joins, for $n>1$, has the same (equivariant) $K$-theory as the complex numbers. This $K$-theoretic computation generalizes a result of Nagy on the free $1$-sphere \cite{g-n94}. 

We remark that there is an equivalent definition of the local-triviality dimension using order zero maps, in a way similar to how the Rokhlin dimension is defined \cite{HWZ15}. This reformulation gives rise to another variant called the \emph{weak local-triviality dimension}. On the other hand, the Rokhlin dimension may also be reformulated in terms of the free noncommutative multi-joins. Thus the aforementioned $K$-theoretic computation may be seen as a reason why we have currently very few means to give lower bounds for the Rokhlin dimension. Moreover, from this perspective, the strong local-triviality dimension is analogous to the Rokhlin dimension with commuting towers (for which we do have obstructions that help give lower bounds). See Section~\ref{rokdimension} for more on these connections between the dimensions as well as a somewhat surprising computation of the Rokhlin dimension for actions by $p$-adic groups on commutative C*-algebras. 

As an evidence that our notion of local triviality dimension behaves in a desired manner, we prove that actions with finite local-triviality dimension are automatically free. We also illustrate our definition by calculating with a variety of examples.

A major motivation for our definition of local triviality is the Borsuk--Ulam-type conjecture. Recall that the classical Borsuk--Ulam theorem says there is no antipodal-equivariant continuous map from an $(n+1)$-sphere to an $n$-sphere, for any $n$. Using the join construction, this theorem is reformulated as: for any $n$, there is no equivariant continuous from $E_n {G} \ast G$ to $E_n {G}$, where $G$ is the two-element group. Analogous statements have been proved for other compact groups $G$. 

In \cite{BDH15}, Baum, D\k{a}browski, and Hajac stated a \emph{noncommutative Borsuk--Ulam-type conjecture}: for any action $\delta$ of a compact quantum group $\mathbb{G}$ on a unital C*-algebra $A$, there is no equivariant $*$-homomorphism from $A \overset{\delta}{\circledast} C(\mathbb{G})$ to $A$, where  $\overset{\delta}{\circledast}$ denotes an equivariant noncommutative join as above. This conjecture, if true even just in the commutative setting, would imply a weak version of the Hilbert-Smith conjecture in topology. On the other hand, progress in the noncommutative setting would give us new insight into the complexity of quantum principal bundles. In fact, from our perspective, this conjecture points to the key difficulty in giving lower bounds on the (strong) local-triviality dimension. 

As an application of our local-triviality dimension, we prove the noncommutative Borsuk--Ulam-type conjecture under the assumption that a compact quantum group admits a non-trivial classical subgroup whose induced action has finite local-triviality dimension. The idea is that the finiteness of the dimension allows us to reduce the problem to one where the compact quantum group $\mathbb{G}$ is replaced by its classical subgroup $H$ and $A$ is replaced by $E_n {H}$, but this latter case is already resolved by classical methods.

{\bf Convention:} All tensor products of C*-algebras are assumed to be minimal. All C*-algebras and $*$-homomorphisms are assumed to be unital unless otherwise stated.

\section{Different types of equivariant noncommutative joins}\label{prem}

\subsection{Preliminaries on joins and actions}

To begin with, we recall the definition of a join (e.g., see \cite[Chapter~0]{a-h02})) of two topological spaces.

\begin{df}
Let $X$ and $Y$ be topological spaces. 
The~{\em topological join} $X\ast Y$ 
of $X$ and $Y$ is defined as the following quotient
\[
X\ast Y:=\left(X\times Y\times [0,1]\right)/\sim\,,
\]
where the equivalence relation is given by
\begin{gather*}
(x,y,1)\sim (x',y,1)\quad\text{for all}\quad x, x'\in X,~y\in Y,\\
(x,y,0)\sim (x,y',0)\quad\text{for all}\quad x\in X,~y, y'\in Y.
\end{gather*}
\end{df}
The topological join construction is associative. It is also functorial in the following sense:
given two continuous maps $f:X\to W$ and $g:Y\to Z$ between topological spaces, 
there exists a continuous map $f\ast g:X\ast Y\to W\ast Z$.

If $X$ and $Y$ are equipped with a continuous free action 
of a topological group $G$, then the diagonal action on the join
is again free. This action is also continuous if $G$ is locally compact Hausdorff 
by a classical result due to Whitehead \cite[Lemma~4]{jhc-w58} (see \cite{e-m68} for a formulation of the result
using modern terminology). This turns $X\ast Y$ into a~free $G$-space.

Recall that, for a locally compact Hausdorff group $G$, the space $G\ast G$ is the first step in the Milnor construction~\cite{j-m56}
of a universal principal $G$-bundle $EG$ with its base space being a model of the classifying space $BG$
(it is true for an arbitrary group $G$ if, instead of the quotient topology, we put the Milnor topology~\cite{j-m56} on the join).
We introduce a concise notation for the multi-join of a topological group~$G$:
\[
E_0G:=G,\qquad E_nG:=\underbrace{G\ast\ldots\ast G}_{n+1}\,,\qquad n>0.
\]

\begin{df}
Let $X$ be a topological space.
The {\it unreduced cone} of $X$ is defined as the quotient
\[
\mathcal{C}X:=\left(X\times [0,1]\right)/\sim\,,
\]
where the equivalence relation is given by 
\[
(x,0)\sim(x',0)\quad\text{for all}\quad x,x'\in X.
\]
\end{df}

We can view the join of $X$ and $Y$ as a subspace of the product of cones:
\begin{equation}\label{jointhecone}
X\ast Y\cong\{([(t,x)],[(s,y)])\in\mathcal{C}X\times\mathcal{C}Y\;|\,t+s=1\}\subseteq \mathcal{C}X\times\mathcal{C}Y.
\end{equation}

Let now $X$ and $Y$ be compact Hausdorff spaces.
One can describe the unital commutative C*-algebras $C(\mathcal{C}X)$ and $C(X\ast Y)$ 
of continuous complex-valued functions on the unreduced
cone and the join respectively, in the following way:
\begin{gather*}
C(\mathcal{C}X)\cong\{f\in C([0,1],C(X))\;|\;f(0)\in\mathbb{C}\},\\
C(X\ast Y)\cong\{f\in C([0,1],C(X)\otimes C(Y))\;|\;f(0)\in C(X)\otimes\mathbb{C},\, f(1)\in \mathbb{C}\otimes C(Y)\}.
\end{gather*}

In the spirit of the celebrated Gelfand--Naimark theorem \cite[Lemma~1]{gn-43} and \cite[Theorems~8', Theorem~10]{i-g41}, 
one can think of unital commutative C*-algebras
as being equivalent to compact Hausdorff topological spaces. Then, the~study of noncommutative C*-algebras can be viewed
as noncommutative topology. Having this in mind, we now recall the unreduced cone and
the generalization of the join construction to the
noncommutative setting.

\begin{df}\label{nccone}
Let $A$ be a unital C*-algebra.
The {\em noncommutative unreduced cone} of $A$ is defined
as follows
\[
\mathcal{C}A:=\{f\in C([0,1],A)\;|\;f(0)\in\mathbb{C}\}.
\]
\end{df}

\begin{df}
Let $A$ and $B$ be unital C*-algebras.
One defines the {\em noncommutative join} $A\circledast B$ of $A$ and $B$ as follows
\[
A\circledast B:=\{f\in C([0,1],A\otimes B)\;|\;f(0)\in A\otimes\mathbb{C},\,f(1)\in \mathbb{C}\otimes B\}.
\]
\end{df}
The noncommutative join construction is associative and functorial. Indeed, given two 
$*$-homo\-mor\-phisms
$\varphi:A\to C$ and $\psi:B\to D$ of unital C*-algebras, there exists a
$*$-homomorphism $\varphi\circledast\psi:A\circledast B\to C\circledast D$.

As in the case of topological spaces, we introduce group symmetry into the picture.
An action of a topological group $G$ on a C*-algebra $A$ is a jointly continuous group homomorphism
$\alpha:G\to {\rm Aut}(A)$ and any C*-algebra equipped with an action of a~group $G$ is called a {\em $G$-C*-algebra}.
We also use the following notation
\[
A^G:=\{a\in A~|~\forall g\in G\,:\,\alpha_g(a)=a\}
\]
for the {\em fixed-point subalgebra} of $A$.
Notice that, for a compact Hausdorff space $X$, there is an isomorphism $C(X)^G\cong C(X/G)$.

One can verify that if $A$ and $B$ are two $G$-C*-algebras, then $A\circledast B$ is again 
a~$G$-C*-algebra with the diagonal action~of~$G$ defined using the functoriality of the join.

In the same way as C*-algebras generalize topological spaces, quantum groups generalize topological groups. Let us now
proceed to actions of compact quantum groups on unital C*-algebras. First, we recall basic definitions.

\begin{df}[\cite{Wo87}]
A {\em compact quantum group} $\mathbb{G}$~is a unital \mbox{C*-al}ge\-bra $C(\mathbb{G})$ 
together with a unital injective $*$-homomorphism
$\Delta:C(\mathbb{G})\to C(\mathbb{G})\otimes C(\mathbb{G})$ that is coassociative, i.e., $({\rm id}\otimes\Delta)
\Delta=(\Delta\otimes{\rm id})\Delta$,
and such that the two-sided cancellation property holds
\[
\{(a\otimes 1)\Delta(b)\;|\;a,b\in C(\mathbb{G})\}^{\rm cls}=C(\mathbb{G})\otimes C(\mathbb{G})
=\{(1\otimes a)\Delta(b)\;|\;a,b\in C(\mathbb{G})\}^{\rm cls},
\]
where ${\rm cls}$ denotes the closed linear span. 
\end{df}

If $G$ is a compact Hausdorff group,
then $C(G)$ is a compact quantum group. This example is the reason for the suggestive
notation $C(\mathbb{G})$ for the C*-algebra of the compact quantum group $\mathbb{G}$.
The group C*-algebra $C^*_r(\Gamma)$ of a discrete group $\Gamma$ is
another example. Here the coproduct is given by $
\Delta(\chi_g)=\chi_g\otimes\chi_g$\,, where $\chi_g$
is the generating unitary corresponding to $g\in\Gamma$.

Note that we assume that the coproduct $\Delta$ is injective. However, to the best of our knowledge, there is no
proof of its injectivity nor an example of a compact quantum group with a non-injective coproduct.

\begin{df}
Let $\mathbb{G}$ be a compact quantum group and let $A$ be a unital C*-algebra. 
A unital injective $*$-homomorphism $\delta:A\to A\otimes C(\mathbb{G})$ is called
a {\em coaction} of $C(\mathbb{G})$ on $A$ (or an {\em action} of $\mathbb{G}$ on $A$)
if and only if
\begin{enumerate}
\item $(\delta\otimes id)\circ\delta=(id\otimes\Delta)\circ\delta$ \quad{\rm (}coassociativity{\rm )},
\item $\{(1\otimes h)\delta(a)\;|\;a\in A, h\in C(\mathbb{G})\}^{\rm cls}=A\otimes C(\mathbb{G})$ \quad {\rm (}counitality{\rm )}.
\end{enumerate}
A C*-algebra $A$ equipped with a coaction of $C(\mathbb{G})$
is called a \emph{$\mathbb{G}$-C*-algebra}. 
\end{df}
In~the case of a coaction, the fixed-point subalgebra is defined as follows
\[
A^\mathbb{G}:=\{a\in A~|~\delta(a)=a\otimes 1\}.
\]

Similarly as for the coproduct, we assume that the coaction $\delta$ is injective,
but here the situation differs as there are examples of non-injective coactions 
(see, e.g. \cite[Proposition~4.1]{Sol11}).
Nevertheless, such coactions in the classical case would correspond to actions
in which the neutral element of the group does not act as the identity.
To exclude these examples, one introduces {\em minimal reduced coactions} \cite{Sol11}. 
Suppose that $\delta:A\to A\otimes C(\mathbb{G})$
is a non-injective action of a compact quantum group $\mathbb{G}$ on a unital C*-algebra $A$.
One defines the minimal reduced coaction by $\overline{\delta}:\overline{A}\to\overline{A}\otimes C(\mathbb{G})$,
where $\overline{A}:=A/\ker\delta$. This coaction is well defined (since $\ker\delta$ is $\mathbb{G}$-invariant)
and injective due to the injectivity of the coproduct \cite[Theorem~3.3]{Sol11}. Throughout the paper,
by a coaction we always mean the minimal reduced coaction.

\begin{df}[\cite{El00}]\label{freeaction}
Let $\delta\colon A\to A\otimes C(\mathbb{G})$ be an action of $\mathbb{G}$ on $A$. 
We say that $\delta$ is \emph{free}~if and only if
\begin{equation}
\{(a\otimes 1_{C(\mathbb{G})})\delta(a)\;|\;a\in A\}^{\rm cls}=A\otimes C(\mathbb{G}).
\end{equation} 
\end{df}

A basic example of a free action in the above sense is the canonical translation action given 
by the coproduct $\Delta\colon C(\mathbb{G})\to C(\mathbb{G})\otimes C(\mathbb{G})$.
Then the Ellwood condition for $\Delta$ is satisfied due to the left-sided cancellation property.
One can show that \eqref{freeaction} generalizes free actions of groups on spaces (e.g., see \cite[Theorem~2.9]{El00}).

Let us now recall the notion of an equivariant noncommutative join of C*-algebras that plays a crucial role in
the noncommutative Borsuk--Ulam-type conjecture \cite{BDH15}. We shall return to this conjecture in Section~\ref{but}.

\begin{df}[\cite{DHH15}]
Let $A$ be a unital $\mathbb{G}$-C*-algebra with a coaction $\delta:A\to A\otimes C(\mathbb{G})$. 
The {\em equivariant noncommutative join} $A\circledast^\delta C(\mathbb{G})$~ of $A$ and $C(\mathbb{G})$ 
is defined as follows
\[
A\overset{\delta}{\circledast} C(\mathbb{G}):=\{f\in C([0,1], A\otimes C(\mathbb{G}))\;|\;f(0)\in \delta(A),f(1)\in\mathbb{C}\otimes C(\mathbb{G})\}.
\]
\end{df}

The above type of a join was introduced as a remedy for the lack of diagonal actions of compact quantum groups
on tensor products of C*-algebras. As an alternative, one could follow \cite{nv10} and use braided tensor products.
In \cite[Corollary~5.6]{BCH17}, it was shown that, given a unital free $\mathbb{G}$-C*-algebra $A$, 
an analog of the diagonal action on $A\circledast^\delta C(\mathbb{G})$ is again free. 

Let $n$ be a nonnegative integer. In analogy with the classical case, we fix the following notation:
\[
C(E_0^{\Delta}\mathbb{G}):=C(\mathbb{G}),\qquad 
C(E_n^{\Delta}\mathbb{G}):=\underbrace{C(\mathbb{G})\overset{\Delta}{\circledast}\ldots\overset{\Delta}{\circledast} C(\mathbb{G})}_{n+1}.
\]

\subsection{The free noncommutative join}\label{freejoin}

In line with the above constructions, we introduce 
a new type of a noncommutative join of C*-algebras using the notion of 
the amalgamated free product $A\ast_D B$ (e.g., see \cite[Definition~II.8.3.5]{b-b06}) 
of unital C*-algebras $A$ and $B$ over their unital $*$-subalgebra~$D$.
This new type of a join will be used to build an $n$-universal locally trivial quantum principal bundle in Section~\ref{qpb}.

\begin{df}\label{freejoin}
Let $A$ and $B$ be unital C*-algebras. The {\em free noncommutative join} of $A$ and $B$ is defined by
\[
A{\mathrlap{+}\times} B:=\{f\in C([0,1],A\ast_{\mathbb{C}} B):f(0)\in A,f(1)\in B\}.
\]
\end{df}

The above construction is associative and functorial 
by the associativity and universality of the amalgamated free product respectively.

Notice that, if $A$ and $B$ are $\mathbb{G}$-C*-algebras with coactions $\delta_A$ and $\delta_B$ respectively,
we can define a {\em diagonal} coaction $\delta_{A\ast_\mathbb{C}B}$. Indeed, let $\iota_A$ and $\iota_B$ be the 
inclusions of $A$ and $B$ in $A\ast_\mathbb{C}B$ respectively. We define the following $*$-homomorphisms
\[
(\iota_A\otimes{\rm id})\circ\delta_A:A\longrightarrow (A\ast_\mathbb{C}B)\otimes C(\mathbb{G}),
\]
\[
(\iota_B\otimes{\rm id})\circ\delta_B:B\longrightarrow (A\ast_\mathbb{C}B)\otimes C(\mathbb{G}).
\]
Then, by the universal property of the amalgamated free product, we obtain the $*$-homomorphism
\[
\delta_{A\ast_\mathbb{C} B}: (A\ast_\mathbb{C}B) \longrightarrow (A\ast_\mathbb{C}B)\otimes C(\mathbb{G}).
\]
It is straightforward to check that $\delta_{A\ast_\mathbb{C} B}$ 
satisfy the coassociativity and counitality conditions. In general, this coaction might not be injective.
In such cases we always consider the minimal reduced coaction.
Using the above coaction, 
one can also define a (minimal reduced) diagonal coaction of $C(\mathbb{G})$ on $A{\mathrlap{+}\times} B$.

We introduce the following notation:
\[
C(E_0^{{\mathrlap{+}\times}}\mathbb{G}):=C(\mathbb{G}),\qquad 
C(E_n^{{\mathrlap{+}\times}}\mathbb{G}):=\underbrace{C(\mathbb{G}){\mathrlap{+}\times}\ldots{\mathrlap{+}\times} C(\mathbb{G})}_{n+1},\qquad n\in\mathbb{N}\setminus\{0\}.
\]

Let us now prove an analogous result to (\ref{jointhecone}) for the free noncommutative join.
Since we were not able to find it in the literature, we first provide a proof in the case of the ordinary
noncommutative join. 

\begin{thm}\label{thmjointhecone2}
Let $A$ and $B$ be unital C*-algebras.
There is an isomoprhism
\begin{equation}\label{jointhecone2}
A\circledast B\cong \left(\mathcal{C}A\otimes\mathcal{C}B\right)/\langle {\rm t}_1\otimes 1+1\otimes {\rm t}_2-1\otimes 1\rangle,
\end{equation}
where ${\rm t}_1$ and ${\rm t}_2$ denote the inclusion of the half-open interval $(0,1]$ into $\mathbb{C}$
in $\mathcal{C}A$ and $\mathcal{C}B$ respectively.
\end{thm}
\begin{proof}
Using Definition~\ref{nccone}, we obtain an isomorphism
\[
\mathcal{C}A\otimes\mathcal{C}B\cong
\{f\in C([0,1]^2,A\otimes B)\;|\;f(t_1,0)\in A\otimes\mathbb{C},\,f(0,t_2)\in\mathbb{C}\otimes B\}.
\]
Observe that $\mathcal{C}A\otimes\mathcal{C}B$ 
has a natural structure of a $C(X)$-algebra, where $X=[0,1]^2$, in the sense of \cite[Definition~1.5]{Kas88}. Consider the ideal
\[
I=\langle F\rangle=\langle {\rm t}_1\otimes 1+1\otimes {\rm t}_2-1\otimes 1\rangle.
\]
Function $F$ vanishes exactly at the points of the standard one-dimensional simplex 
$\Delta^1=\{(s,t)\in[0,1]^2\;|\:s+t=1\}$.
Therefore, we have that
\begin{align*}
(\mathcal{C}A\otimes\mathcal{C}B)/I
&\cong\{f\in C(\Delta^1,A\otimes B)\;|\;f(t_1,0)\in A\otimes\mathbb{C},\,f(0,t_2)\in\mathbb{C}\otimes B\}\\
&\cong\{f\in C([0,1],A\otimes B)\;|\;f(0)\in A\otimes\mathbb{C},\,f(1)\in \mathbb{C}\otimes B\}=A\circledast B.
\end{align*}
\end{proof}

Next, we prove a similar result for the free join $A{\mathrlap{+}\times}B$.
In fact we prove even more, namely that $A{\mathrlap{+}\times} B$ is also isomorphic to the amalgamated free product
$\mathcal{C}A\ast_{C([0,1])}\mathcal{C}B$ by means of
the unital $*$-homomorphisms given on generators by
\begin{equation}\label{t1t2}
f_1:C([0,1])\to \mathcal{C}A\;:\;{\rm t}\mapsto 1-{\rm t}_1,\quad f_2:C([0,1])\to \mathcal{C}B\;:\; {\rm t}\mapsto {\rm t}_2.
\end{equation}
Here ${\rm t}$ is the inclusion function from $(0,1]$ into $\mathbb{C}$ generating $C_0((0,1])$ as a \mbox{C*-al}gebra
and ${\rm t}_i$\,, $i=1,2$, are defined as in Theorem~\ref{thmjointhecone2}.

\begin{thm}\label{jointhecone3}
Let $A$ and $B$ be unital C*-algebras. We have the following 
isomorphisms
\[
A{\mathrlap{+}\times} B\cong
\left(\mathcal{C}A\ast_\mathbb{C} \mathcal{C}B\right)/\langle {\rm t}_1+{\rm t}_2-1\rangle\cong 
\mathcal{C}A\ast_{C([0,1])}\mathcal{C}B.
\]
\end{thm}
\begin{proof} 
First, note that the second isomorphism follows from the universal properties of $\mathcal{C}A\ast_\mathbb{C} \mathcal{C}B$
and $\mathcal{C}A\ast_{C([0,1])}\mathcal{C}B$. Therefore, it suffices to show that $A{\mathrlap{+}\times}B$ is isomorphic to
$\mathcal{C}A\ast_{C([0,1])}\mathcal{C}B$. 

Define two $*$-homomorphisms
\[
\psi_1:\mathcal{C}A\to A {\mathrlap{+}\times} B,\quad (\psi_1(a))(t):=a(1-t),
\]
\[
\psi_2:\mathcal{C}B\to A {\mathrlap{+}\times} B,\quad (\psi_1(b))(t):=b(t).
\]
Let $f_1$ and $f_2$ be defined as in \eqref{t1t2}. Then,
$\psi_1\circ f_1=\psi_2\circ f_2$,
and by the universal property of $\mathcal{C}A\ast_{C([0,1])}\mathcal{C}B$, the maps $\psi_1$ and
$\psi_2$ give rise to a unital $*$-homomorphism
\[
\psi:\mathcal{C}A\ast_{C([0,1])}\mathcal{C}B\to A {\mathrlap{+}\times} B.
\]

We will show that the above map is an isomorphism. For this purpose, observe that
both the domain and the codomain of $\psi$ can be viewed as $C(X)$-algebras
(see \cite[Definition~1.5]{Kas88}), where $X=[0,1]$.
Indeed, if we denote by $\iota$ and $\iota'$ the inclusions of $C([0,1])$ 
in $\mathcal{C}A\ast_{C([0,1])}\mathcal{C}B$ and 
$A {\mathrlap{+}\times} B$ respectively, their images land in the centers, namely 
\begin{equation}\label{incent}
\iota(C([0,1]))\subseteq Z(\mathcal{C}A\ast_{C([0,1])}\mathcal{C}B),\quad\iota'(C([0,1]))\subseteq Z(A {\mathrlap{+}\times} B).
\end{equation}

Any isomorphism of $C(X)$-algebras over the same space is given by an
isomorphism of the C*-algebras over the fibers via the fiber-preserving map.
Let us now describe the fibers of the $C(X)$-algebras under consideration.
Fix a $t_0\in [0,1]$. One defines
\[
I_{t_0}:=\iota(C_0([0,1]\setminus\{t_0\}))\cdot\left(\mathcal{C}A\ast_{C([0,1])}\mathcal{C}B\right)
\]
and subsequently
\[
\left.\left(\mathcal{C}A\ast_{C([0,1])}\mathcal{C}B\right)\right|_{t_0}
:=\left(\mathcal{C}A\ast_{C([0,1])}\mathcal{C}B\right)/I_{t_0}\,.
\]
This quotient is well defined because of the first inclusion in \eqref{incent}.
Taking advantage of \eqref{t1t2}, we observe that
\[
\left.\left(\mathcal{C}A\ast_{C([0,1])}\mathcal{C}B\right)\right|_{t_0}\cong 
\left.\mathcal{C}A\right|_{1-t_0}\ast_\mathbb{C}\left.\mathcal{C}B\right|_{t_0}\cong
\begin{cases}A & t_0=0\\ A\ast_\mathbb{C}B & t_0\in (0,1)\\B & t_0=1\end{cases}.
\]

Similarly, one can define $I'_{t_0}$ (using the inclusion $\iota'$) and $\left.\left(A {\mathrlap{+}\times} B\right)\right|_{t_0}$\,. 
It is straightforward to see that
\[
I'_{t_0}\cong\{f\in C([0,1],A\ast_\mathbb{C}B)~|~f(0)\in A, f(1)\in B, f(t_0)=0\}
\]
and subsequently
\[
\left.\left(A {\mathrlap{+}\times} B\right)\right|_{t_0}\cong
\begin{cases}A & t_0=0\\ A\ast_\mathbb{C}B & t_0\in (0,1)\\B & t_0=1\end{cases}.
\]

Now, $\psi$ induces a $*$-homomorphism between fibers
\[
\psi_{t_0}:\left.\mathcal{C}A\right|_{1-t_0}\ast_\mathbb{C}\left.\mathcal{C}B\right|_{t_0}
\to\left.\left(A {\mathrlap{+}\times} B\right)\right|_{t_0}
\]
and it suffices to show that $\psi_{t_0}$ is an isomorphism for every $t_0\in[0,1]$.
This can be done using a standard partition-of-unity argument for $[0,1]$.
\end{proof}
\begin{ex}[Free spheres]\label{frees}
The $n$-dimensional free sphere $C(S^n_+)$ is defined as the unital universal C*-algebra generated by $n+1$
elements $x_0$, $x_1$, $\ldots$, $x_n$ subject to relations
$$
x^*_i=x_i,\quad i=0,\ldots,n,\qquad\sum_{i=0}^nx^2_i=1.
$$
Our aim is to show that any free noncommutative sphere is isomorphic to an iterated free noncommutative join 
of $C(\mathbb{Z}/2\mathbb{Z})$, namely that
$$
C(S^n_+)\cong C(E_n^{\mathrlap{+}\times}\mathbb{Z}/2\mathbb{Z}).
$$
By Theorem \ref{jointhecone3}, we have that
\[
C(E_n^{\mathrlap{+}\times}\mathbb{Z}/2\mathbb{Z})\cong 
(\,\underbrace{\mathcal{C}C(\mathbb{Z}/2\mathbb{Z})\ast_\mathbb{C}\ldots
\ast_\mathbb{C}\mathcal{C}C(\mathbb{Z}/2\mathbb{Z}\,)}_{n+1})/I,
\]
where $I$ is the ideal generated by $1-\sum_{i=0}^n {\rm t}_i$. Here each ${\rm t}_i$ denotes the inclusion of the half-open 
interval $(0,1]$ into $\mathbb{C}$ in a different copy of $\mathcal{C}C(\mathbb{Z}/2\mathbb{Z})$. Note that
$$
\mathcal{C}C(\mathbb{Z}/2\mathbb{Z})\cong C(\mathcal{C}\mathbb{Z}/2\mathbb{Z})\cong C([-1,1]),
$$
where we view $C([-1,1])$ as the unital universal algebra generated by one self-adjoint element $p$ such that $\|p\|\leq 1$.
There is a $\mathbb{Z}/2\mathbb{Z}$-action on $C([-1,1])$ given by $p\mapsto -p$.
If we denote by ${\rm t}$ the inclusion of the half-open interval $(0,1]$ into $\mathbb{C}$ in $\mathcal{C}C(\mathbb{Z}/2\mathbb{Z})$,
then $p^2={\rm t}$.

We have a natural well-defined $\mathbb{Z}/2\mathbb{Z}$-equivariant $*$-homomorphism
\begin{gather*}
\Phi:C(S^n_+)\to \left(C([-1,1])^{\ast_\mathbb{C}n}\right)/\left\langle \Sigma_{i=0}^n\; p^2_i-1\right\rangle,\\
x_i\mapsto p_i,\qquad i=0,1,\ldots,n,
\end{gather*}
where each $p_i$ is the generator of the $i$th copy of $C([-1,1])$ in the iterated amalgamated free product.
By the universality of the free product, we find the inverse of the above map and we conclude the isomorphism.\hfill$\diamond$
\end{ex}

We end this section by showing that the $K$-theory of any free noncommutative join is the same as that of $\mathbb{C}$. 
\begin{thm}
Let $A$ and $B$ be unital C*-algebras. The embedding 
$\mathbb{C}\hookrightarrow A{\mathrlap{+}\times}B$
induces isomorphisms of the $K$-theory groups:
\[
K_0(A{\mathrlap{+}\times}B)=\mathbb{Z},\qquad K_1(A{\mathrlap{+}\times}B)=0,
\]
where $K_0(A{\mathrlap{+}\times}B)$ is generated by $[1]$.
\end{thm}
\begin{proof}
Without loss of generality, we restrict to separable C*-algebras, since
every C*-algebra is a direct limit of separable C*-algebras and $K$-theory
functor is continuous with respect to direct limits. 
Recall that, by Theorem \ref{jointhecone3}, we have the isomorphism
\[
A{\mathrlap{+}\times}B\cong \mathcal{C}A\ast_{C([0,1])}\mathcal{C}B.
\]
By \cite[Theorem 6.4]{Thom03}, there is the following six-term exact sequence:
\[
\xymatrix{
K_0(C([0,1])) \ar[r]^{((f_1)_*,(f_2)_*)}
& K_0(\mathcal{C}A)\oplus K_0(\mathcal{C}B) \ar[r]
& K_0(A{\mathrlap{+}\times}B) \ar[d]\\
K_1(A{\mathrlap{+}\times}B)\ar[u]
& K_1(\mathcal{C}A)\oplus K_1(\mathcal{C}B) \ar[l]
& K_1(C([0,1]))~~ \ar[l]}
\]
from which the statement follows, since $K_0(C([0,1]))=K_0(\mathcal{C}A)=K_0(\mathcal{C}B)=\mathbb{Z}$
(generated by $[1]$) and $K_1(C([0,1]))=K_1(\mathcal{C}A)=K_1(\mathcal{C}B)=0$ for any $A$ and $B$.
\end{proof}
Let us remark, that the above result is also true at the level of the equivariant $K$-theory.

As a corollary, we obtain the $K$-theory of free noncommutative spheres $C(S^n_+)$. 
\begin{cor}
Let $C(S^n_+)$ be the noncommutative free sphere. We have that
\[
K_0(C(S^n_+))=\mathbb{Z},\qquad K_1(C(S^n_+))=0,
\]
where $K_0(C(S^n_+))$ is generated by $[1]$.
\end{cor}
As far as the authors know,
for $n>1$ (see \cite{g-n94} for $n=1$), this result was not stated in the literature before.

\section{Locally trivial compact quantum principal bundles}\label{qpb}

In this section, we introduce three different incarnations of the local-triviality dimension for 
actions of compact 
quantum groups on unital C*-algebras and introduce the concept of a locally trivial 
compact quantum principal bundle.
We also prove that local triviality in that sense implies freeness of the action and 
introduce $n$-universal locally trivial compact
quantum principal bundles. 

Throughout this section ${\rm t}$ denotes 
the inclusion function from the half-open interval $(0,1]$ into
$\mathbb{C}$, which generates $C_0((0,1])$ as a C*-algebra.

\subsection{The (weak) local-triviality dimension}

\begin{df}\label{lt}
Let $A$ be a unital C*-algebra equipped with an action $\delta$
of a compact quantum group $\mathbb{G}$.
If $d$ is the minimal nonnegative integer such that there exist
\mbox{$\mathbb{G}$-equi}variant~$*$-homomorphisms $$\rho_0,\ldots, \rho_d :C_0((0,1])\otimes C(\mathbb{G})\to A$$
satisfying the condition that:
\begin{enumerate}
\item $\sum_{j=0}^d\rho_j({\rm t}\otimes 1)$ is~invertible, we say that $\delta$ has the \emph{weak local-triviality dimension} $d$, written $\dim_{\rm{WLT}}(\delta)=d$, and we set $\dim_{\rm{WLT}}(\delta)=\infty$ if no such $d$ exists,
\item $\sum_{j=0}^d\rho_j({\rm t}\otimes 1)=1$ (joint-unitality), we say that $\delta$ has the \emph{local-triviality dimension} $d$, written $\dim_{\rm{LT}}(\delta)=d$, and we set $\dim_{\rm{LT}}(\delta)=\infty$ if no such $d$ exists.
\end{enumerate}
\end{df}

\begin{notation}
If there is no ambiguity about the $\mathbb{G}$-action on $A$,
we will also denote the (weak) local-triviality dimension by 
$\dim_{\rm (W)LT}^\mathbb{G}(A)$.\hfill$\diamond$
\end{notation}

It is immediate from Definition~\ref{lt} that for any coaction $\delta$, we have that
\begin{equation}\label{wltlesslt}
\dim_{\rm{WLT}}(\delta)\leq \dim_{\rm{LT}}(\delta)\,.
\end{equation}
Next, let $A$ and $B$ be $\mathbb{G}$-C*-algebras with actions $\delta_A$ and $\delta_B$ 
respectively. 
If there exists a $\mathbb{G}$-equivariant $*$-homomorphism $A\to B$, then
\begin{equation}\label{loceq}
\dim_{\rm WLT}(\delta_A)\geq\dim_{\rm WLT}(\delta_B),\qquad\dim_{\rm LT}(\delta_A)\geq\dim_{\rm LT}(\delta_B).
\end{equation}
The above property of the local-triviality dimensions is useful for Borsuk--Ulam-type problems (see Section~\ref{but}),
where one usually tries to establish nonexistence of certain equivariant maps. Moreover, if
$A$ is a unital $\mathbb{G}$-C*-algebra with an action $\delta_A$ and $I$ is a {\em $\mathbb{G}$-invariant} ideal, 
i.e. $\delta_A(I)\subseteq I\otimes C(\mathbb{G})$, then it follows from \eqref{loceq} that
\begin{equation}\label{quotient}
\dim_{\rm WLT}(\delta_{A/I})\leq \dim_{\rm WLT}(\delta),\qquad \dim_{\rm LT}(\delta_{A/I})\leq \dim_{\rm LT}(\delta).
\end{equation}
Here $\delta_{A/I}$ is the action induced on the quotient by $\delta_A$.

\begin{rem}
Few comments regarding the minimal reduced coactions are in order.
Recall that if $\delta_A:A\to A\otimes C(\mathbb{G})$ is not injective
we use $\overline{\delta}_A:\overline{A}\to\overline{A}\otimes C(\mathbb{G})$ instead,
where $\overline{A}=A/\ker\delta$. 
Adapting our definition of the local-triviality dimension to non-injective coactions, by \eqref{quotient}, we obtain
\[
\dim_{\rm LT}(\overline{\delta}_A)\leq \dim_{\rm LT}(\delta_A).
\]
Hence considering the minimal coaction does not change the finiteness of the local-triviality dimension.
Next, let $B$ be another $\mathbb{G}$-C*-algebra with a non-injective coaction $\delta_B$
and suppose that there is a $\mathbb{G}$-equivariant $*$-homomorphism $A\to B$.
Then, there is a $\mathbb{G}$-equivariant $*$-homomorphism $\overline{A}\to \overline{B}$,
so that an analog of \eqref{loceq} is satisfied.
\hfill$\diamond$
\end{rem}

We list some elementary facts about the local-triviality dimensions.
\begin{prop}\label{trivializable}
Let $\delta:A\to A\otimes C(\mathbb{G})$ be an action of
a compact quantum group $\mathbb{G}$ on a unital C*-algebra $A$.
Then $\dim_{\rm LT}(\delta)=0$ if and only if there exists
a~$\mathbb{G}$-equivariant *-homomorphism $C(\mathbb{G})\to A$.
\end{prop}
\begin{proof}
Let $\dim_{\rm LT}(\delta)=0$, so we have a
$*$-$\mathbb{G}$-homomorphism $\rho:C_0((0,1])\otimes C(\mathbb{G})\to A$
such that $\rho({\rm t}\otimes 1)=1$. Consider the map
\[
\varphi:C(\mathbb{G})\longrightarrow A: h\longmapsto \rho({\rm t}\otimes h).
\]
The above assignment is clearly unital, $\mathbb{G}$-equivariant and respects the $*$-structure. 
Moreover, since
\begin{align*}
\varphi(ab)&=\rho({\rm t}\otimes ab)
=\rho(\sqrt{{\rm t}}\otimes a)\rho({\rm t}\otimes 1)\rho(\sqrt{{\rm t}}\otimes b)
=\rho({\rm t}\otimes a)\rho({\rm t}\otimes b)=\varphi(a)\varphi(b),
\end{align*}
it is also an algebra homomorphism. The other implication is analogous.
\end{proof}

\begin{prop}\label{dimzero}
Let $\delta:A\to A\otimes C(\mathbb{G})$ be an action of
a compact quantum group $\mathbb{G}$ on a unital C*-algebra $A$.
Then we have that
\[
\dim_{\rm{WLT}}(\delta)=0 \iff \dim_{\rm{LT}}(\delta)=0.
\]
\end{prop}
\begin{proof}
For the not immediate implication,
let $\dim_{\rm{WLT}}(\delta)=0$, so that 
we have \mbox{a~$\mathbb{G}$-equi}variant *-homomorphism  
$\rho:C_0((0,1])\otimes C(\mathbb{G})\to A$, such that $\rho({\rm t}\otimes 1)$ is invertible.
Every element of $C_0((0,1])\otimes C(\mathbb{G})$ can be approximated by
linear combinations of elements of the form ${\rm t}^k\otimes h$, where 
$k\in\mathbb{N}\setminus\{0\}$ and $h\in C(\mathbb{G})$.
Now consider a map
$\widetilde{\rho}:C_0((0,1])\otimes C(\mathbb{G})\to A$ defined by
\[
\widetilde{\rho}({\rm t}^k\otimes h)=
\rho(\sqrt{{\rm t}}\otimes 1)^{-k}\rho({\rm t}\otimes h)\rho(\sqrt{{\rm t}}\otimes 1)^{-k}
\] 
for any $k\in\mathbb{N}\setminus\{0\}$ and $h\in C(\mathbb{G})$.
Since $\rho({\rm t}\otimes 1)$ is invertible and $\mathbb{G}$-invariant,
this map is well defined and $\mathbb{G}$-equivariant. For any $g, h\in C(\mathbb{G})$, we have that
\begin{align*}
\widetilde{\rho}({\rm t}\otimes g)\widetilde{\rho}({\rm t}\otimes h)&=
\rho(\sqrt{{\rm t}}\otimes 1)^{-1}\rho({\rm t}\otimes g)\rho(\sqrt{{\rm t}}\otimes 1)^{-1}
\rho(\sqrt{{\rm t}}\otimes 1)^{-1}\rho({\rm t}\otimes h)\rho(\sqrt{{\rm t}}\otimes 1)^{-1}\\
&=\rho(\sqrt{{\rm t}}\otimes 1)^{-1}\rho(\sqrt{{\rm t}}\otimes g)
\rho(\sqrt{{\rm t}}\otimes h)\rho(\sqrt{{\rm t}}\otimes 1)^{-1}\\
&=\rho(\sqrt{{\rm t}}\otimes h)^{-2}\rho({\rm t}^2\otimes gh)\rho(\sqrt{{\rm t}}\otimes 1)^{-2}\\
&=\widetilde{\rho}({\rm t}^2\otimes gh).
\end{align*}
Hence, $\widetilde{\rho}$ can be extended to a $\mathbb{G}$-equivariant $*$-homomorphism
\[
\widetilde{\rho}:C_0((0,1])\otimes C(\mathbb{G})\to A.
\]
One can verify that $\widetilde{\rho}({\rm t}\otimes 1)=1$
and therefore $\dim_{\rm LT}(\delta)=0$.
\end{proof}
\begin{cor}
Let $\delta:A\to A\otimes C(\mathbb{G})$ be an action of
a compact quantum group $\mathbb{G}$ on a unital C*-algebra $A$.
Then $\dim_{\rm LT}(\delta)=1$ implies that $\dim_{\rm WLT}(\delta)=1$.
\end{cor}
\begin{proof}
Let $\dim_{\rm LT}(\delta)=1$. This implies that $\dim_{\rm WLT}(\delta)\leq 1$
by \eqref{wltlesslt}.
Since $\dim_{\rm LT}(\delta)\neq 0$ implies that $\dim_{\rm WLT}(\delta)\neq 0$ 
by Proposition~\ref{dimzero}, the claim follows.
\end{proof}

\begin{prop}
Let $\delta:A\to A\otimes C(\mathbb{G})$ be an action of
a compact quantum group $\mathbb{G}$ on a unital C*-algebra $A$.
Suppose that $\dim_{\rm WLT}(\delta)=d$ and that $\sum_{j=0}^d\rho_j({\rm t}\otimes 1)$ commutes
with the images of the maps $\rho_j$ for all $j=0,1,\ldots,d$. Then,
$\dim_{\rm{LT}}(\delta)=d$.
\end{prop}
\begin{proof}
Suppose that $\dim_{\rm WLT}(\delta)=d$ and let
$e:=\sum_{j=0}^d\rho_j({\rm t}\otimes 1)$.
Similarly as in the proof of Proposition \ref{dimzero}, we define the maps
\[
\widetilde{\rho}_i({\rm t}^k\otimes h):=\rho_i({\rm t}^k\otimes h)e^{-k},\qquad i=0,1,\ldots, d,
\]
and, since $e$ is invertible, $\mathbb{G}$-invariant, and it
commutes with the images of the maps $\rho_j$,  
we extend each $\widetilde{\rho}_i$ to a $\mathbb{G}$-equivariant 
$*$-homomorphism
\[
\widetilde{\rho}_i:C_0((0,1])\otimes C(\mathbb{G})\to A, \qquad i=0,1,\ldots, d.
\] 
It only remains to check that the joint-unitality condition is satisfied:
\[
\sum_{i=0}^d\widetilde{\rho}_i({\rm t}\otimes 1)=\sum_{i=0}^d\rho_i({\rm t}\otimes 1)e^{-1}=ee^{-1}=1.
\]
\end{proof}

Next, we show that finiteness of the weak local-triviality dimension implies freeness in the sense of Definition \ref{freeaction}, 
in complete generality.

\begin{thm}\label{thm:LocTrivFree}
Let $\mathbb{G}$ be a compact quantum group, let $A$ be a unital C*-algebra, and let $\delta$ be an action
of $\mathbb{G}$ on $A$. If $\dim_{\rm WLT}(\delta)<\infty$, then $\delta$ is free.
\end{thm}
\begin{proof}
Set $B=\{(A\otimes 1_{C(\mathbb{G})})\delta(A)\}^{\rm cls}$. To show that $B=A\otimes C(\mathbb{G})$, it is
enough to show that $B$ contains all simple tensors. In addition, since $B$ is a left $(A\otimes 1_{C(\mathbb{G})})$-module,
it suffices to show that $1_A\otimes x$ belongs to $B$ for all $x\in C(\mathbb{G})$.

Let $x\in C(\mathbb{G})$ be fixed, and set $d=\dim_{\rm WLT}(\delta)<\infty$. Let $\varepsilon>0$. We write $f\approx_\varepsilon g$ to mean that $\|f-g\|<\varepsilon$.
Since the action $\Delta\colon C(\mathbb{G})\to C(\mathbb{G})\otimes C(\mathbb{G})$ is free,
there are $m\in\mathbb{N}$ and $y_1,\ldots, y_m,z_1,\ldots,z_m\in C(\mathbb{G})$ such that
\[1_{C(\mathbb{G})}\otimes x\approx_{\varepsilon}\sum_{k=1}^m ( y_k\otimes 1_{C(\mathbb{G})})\Delta(z_k).\]
Using the fact that $\dim_{\rm WLT}(\delta)=d$, we find
$\mathbb{G}$-equivariant $*$-homomorphisms $$\rho_0,\ldots, \rho_d \colon C_0((0,1])\otimes C(\mathbb{G})\to A$$
such that the element $e:=\sum_{j=0}^d\rho_j({\rm t}\otimes 1_{C(\mathbb{G})})$ is invertible.
For $j=0,\ldots,d$, set $\widetilde{\rho}_j=\rho_j \otimes {\rm id}_{C(\mathbb{G})}\colon C_0((0,1])\otimes C(\mathbb{G})\otimes C(\mathbb{G})
\to A\otimes C(\mathbb{G})$, which are also $\mathbb{G}$-equivariant $*$-homomorphisms. Then,
\[
\begin{split}
e\otimes x &=\sum\limits_{j=0}^d\rho_j({\rm t}\otimes 1_{C(\mathbb{G})})\otimes x=\sum\limits_{j=0}^d\widetilde{\rho}_j({\rm t}\otimes 1_{C(\mathbb{G})}\otimes x)\\ &\approx_{\varepsilon}\sum\limits_{j=0}^d\sum\limits_{k=1}^m\widetilde{\rho}_j( {\rm t}\otimes (y_k\otimes 1_{C(\mathbb{G})})\Delta(z_k))\\
&= \sum\limits_{j=0}^d\sum\limits_{k=1}^m\widetilde{\rho}_j( \sqrt{{\rm t}}\otimes (y_k\otimes 1_{C(\mathbb{G})}))\,
\widetilde{\rho}_j( \sqrt{{\rm t}}\otimes\Delta(z_k))\\
&=\sum\limits_{j=0}^d\sum\limits_{k=1}^m(\rho_j(\sqrt{{\rm t}}\otimes y_k)\otimes 1_{C(\mathbb{G})}) \delta(\rho_j(\sqrt{{\rm t}}\otimes z_k)),
\end{split}
\]
where we used the equivariance of the maps $\rho_j$.

This shows that $e\otimes x$ belongs to $B$. Since $e$ is invertible, $1_A\otimes x$ belongs to $B$ as well. 
Hence $B=A\otimes C(\mathbb{G})$ and we conclude that $\delta$ is free.
\end{proof}

Next, we introduce the notion of a locally trivial compact quantum principal bundle.
Let $\delta:A\to A\otimes C(\mathbb{G})$ be a free
action of a compact quantum group $\mathbb{G}$
on a~unital C*-algebra $A$.
Then the triple $(A,A^\mathbb{G},\mathbb{G})$ is called a
{\em compact quantum principal bundle} 
(cf. \cite[Definition~3.1]{BDH15}). 
Hence, by Theorem~\ref{thm:LocTrivFree}, 
any action of a compact quantum group $\mathbb{G}$ with finite local-triviality dimension
gives rise to a~compact quantum principal bundle.
We arrive at the following definition, which is a~noncommutative analog
of a~locally trivial compact principal bundle (see Section~\ref{classical}).

\begin{df}
A compact quantum principal bundle $(A,A^\mathbb{G},\mathbb{G})$
is said to be {\em locally trivial} if and only if $\dim_{\rm LT}^{\mathbb{G}}(A)<\infty$.
\end{df}

A compact quantum principal 
bundle $(A, A^\mathbb{G}, \mathbb{G})$ is called 
{\em trivializable}
if there exists a $\mathbb{G}$-equivariant $*$-homomorphism $C(\mathbb{G})\to A$
\cite[Definition~3.1]{BDH15}. Hence, by Proposition~\ref{trivializable},
$(A, A^\mathbb{G}, \mathbb{G})$ is trivializable if and only if $\dim_{\rm LT}^\mathbb{G}(A)=0$.

\subsection{Examples of locally trivial compact quantum principal bundles}

We start with antipodal 
$\mathbb{Z}/2\mathbb{Z}$-actions on two different kinds of noncommutative spheres
and show that they have finite local-triviality dimension.
\begin{ex}[Antipodal action on the free spheres]\label{antifreesphere}
There is a natural antipodal action of $\mathbb{Z}/2\mathbb{Z}$ on $C(S^n_+)$ (see Example \ref{frees}) given on generators by
\[
x_i\mapsto -x_i,\qquad i=0,1,2,\ldots,n.
\]
Recall that $C(\mathbb{Z}/2\mathbb{Z})$ can be viewed as the universal C*-algebra generated by a single self-adjoint element
$\gamma$ such that $\gamma^2=1$. One can find equivariant $*$-homomorphisms
\[
\varphi_i:C_0((0,1])\otimes C(\mathbb{Z}/2\mathbb{Z})\to C(S^n_+):\sqrt{{\rm t}}\otimes\gamma\mapsto x_i,\qquad i=0,1,2,\ldots,n.
\]
Observe that
\[
\sum_{i=0}^n\varphi_i({\rm t}\otimes 1)=\sum_{i=0}^n\left(\varphi_i(\sqrt{{\rm t}}\otimes\gamma)\right)^2=\sum_{i=0}^nx^2_i=1.
\]
The above considerations imply that $\dim_{\rm LT}^{\mathbb{Z}/2\mathbb{Z}}(C(S^n_+))\leq n$.\hfill$\diamond$
\end{ex}
\begin{ex}[Antipodal action on the equatorial Podle\'s sphere]
Recall that the equatorial 
Podle\'s sphere $C(S^2_{q\infty})$ \cite{Pod87} is the unital universal 
C*-algebra generated by $B$ and self-adjoint $A$ satisfying the relations
\[
AB=q^2BA,\qquad B^*B=1-A^2,\qquad BB^*=1-q^4A^2.
\]
The antipodal $\mathbb{Z}/2\mathbb{Z}$-action is defined by
\[
A\mapsto -A,\qquad B\mapsto -B.
\]
Define $\mathbb{Z}/2\mathbb{Z}$-equivariant $*$-homomorphisms
\begin{gather*}
\rho_0:C_0((0,1])\otimes C(\mathbb{Z}/2\mathbb{Z})\to C(S^2_q):\sqrt{{\rm t}}\otimes\gamma\mapsto \frac{i}{2}(B^*-B),\\
\rho_1:C_0((0,1])\otimes C(\mathbb{Z}/2\mathbb{Z})\to C(S^2_q):\sqrt{{\rm t}}\otimes\gamma\mapsto \frac{1}{2}(B^*+B),\\
\rho_2:C_0((0,1])\otimes C(\mathbb{Z}/2\mathbb{Z})\to C(S^2_q):\sqrt{{\rm t}}\otimes\gamma\mapsto \sqrt{\frac{1+q^4}{2}}A.
\end{gather*}
Using the defining relations of $C(S^2_{q\infty})$, one can check that
\[
\sum_{i=0}^2\rho_i({\rm t}\otimes 1)=\sum_{i=0}^2\left[\rho_i(\sqrt{{\rm t}}\otimes\gamma)\right]^2=\frac{1}{2}B^*B+\frac{1}{2}BB^*+\frac{1+q^4}{2}A^2=1. 
\]
Hence, $\dim_{\rm LT}^{\mathbb{Z}/2\mathbb{Z}}(C(S^2_{q\infty}))\leq 2$.\hfill$\diamond$
\end{ex}

Note that finiteness of the local-triviality dimension of any 
$\mathbb{Z}/2\mathbb{Z}$-action on a~unital
C*-algebra $A$ is tantamount to the existence of finitely many odd self-adjoint elements
in $A$ whose squares add up to one.
We will show in Section~\ref{but}, that $\dim_{\rm LT}^{\mathbb{Z}/2\mathbb{Z}}(C(S^n_+))=n$ for all $n$.

In the next example we consider a $U(1)$-action for which we not only bound 
the local-triviality dimension but also obtain its actual value. Then, we present
an example of an $SU(2)$-action of a similar flavour.

\begin{ex}[Noncommutative Matsumoto--Hopf fibration]
Let us consider the Matsumoto noncommutative three-sphere $S^3_\theta$ \cite{Mts91,Mts91II}, where $\theta\in(0,1)$. It is 
defined as the universal C*-algebra generated by two normal elements $Z$ and $W$ subject to relations
\[
ZW=e^{2\pi i\theta}WZ,\qquad Z^*Z+W^*W=1.
\]
Define the action of $U(1)$ on $C(S^3_\theta)$ by
\[
Z\mapsto e^{i\varphi}Z,\qquad W\mapsto e^{i\varphi}W,\qquad \mbox{for all}\quad e^{i\varphi}\in U(1).
\]
Since Matsumoto showed that $C(S^3_\theta)^{U(1)}\cong C(S^2)$ \cite{Mts91II}, 
the action considered in this example is sometimes called the 
noncommutative Matsumoto--Hopf fibration (or the~noncommutative Matsumoto--Dirac 
monopole bundle).

Recall that $C(U(1))$ can be viewed as the universal C*-algebra generated by one unitary $U$. Define two 
$U(1)$-equivariant $*$-homomorphisms
\[
\rho_0:C_0((0,1])\otimes C(U(1))\to C(S^3_\theta): \sqrt{{\rm t}}\otimes U\mapsto Z,
\]
\[
\rho_1:C_0((0,1])\otimes C(U(1))\to C(S^3_\theta): \sqrt{{\rm t}}\otimes U\mapsto W.
\]
Note that $\rho_0({\rm t}\otimes 1)+\rho_1({\rm t}\otimes 1)=1$, and hence $\dim_{\rm LT}^{U(1)}(C(S^3_\theta))\leq 1$.

Assume now that $\dim_{\rm LT}^{U(1)}(C(S^3_\theta))= 0$. 
By Proposition~\ref{trivializable}, there is a $U(1)$-equivariant $*$-homo\-mor\-phism 
$C(U(1))\to C(S^3_\theta)$. This map is in particular $\mathbb{Z}/2\mathbb{Z}$-equivariant
and the image of $U$ under this map would contradict \cite[Proposition~3.9]{b-p16}
(which states that no element of $C(S^3_\theta)$ can be both odd and invertible).
Hence, we have that $\dim_{\rm LT}^{U(1)}(C(S^3_\theta))=1$.
\hfill$\diamond$
\end{ex}
Similarly as for the case of $\mathbb{Z}/2\mathbb{Z}$, 
the finiteness of the local-triviality dimension
of any $U(1)$-action on a C*-algebra $A$ can be stated in a different way.
Recall that any $U(1)$-action gives a $\mathbb{Z}$-grading on $A$.
Then the fact that $\dim_{\rm LT}^{U(1)}(A)<\infty$ can be translated into 
existence of finitely
many normal elements $x_i$ of degree $1$ in $A$ 
such that $\sum_i x^*_ix_i=1$.

\begin{ex}[Noncommutative Hopf principal $SU(2)$-bundle]
Let $\theta=(\theta_{ij})$ be an $n\times n$ Hermitian matrix such that $|\theta_{ij}|=1$ for all $i,j$
and with $\theta_{ii}=1$ for each $i$.
The C*-algebra $C(S^{2n-1}_\theta)$ of the {\em odd Natsume--Olsen quantum sphere} 
$S^{2n-1}_\theta$ \cite{no97} is defined as the universal unital C*-algebra generated by
$n$ normal elements $Z_1$, \dots, $Z_n$ satisfying the relations
\[
\sum_{i=1}^nZ^*_iZ_i=1,\qquad Z_jZ_i=\theta_{ij} Z_iZ_j,\qquad i,j=1,2,\ldots, n.
\]

Recall that $C(SU(2))$ can be viewed as the unital universal C*-algebra
generated by two normal elements $\alpha$ and $\gamma$ commuting with each other and
satisfying
\[
\alpha^*\alpha+\gamma^*\gamma=1.
\]
We write the coproduct of $C(SU(2))$ on generators as follows
\[
\Delta(\alpha)=\alpha\otimes\alpha-\gamma\otimes\alpha^*,\qquad 
\Delta(\gamma)=\alpha\otimes\gamma+\gamma\otimes\gamma^*.
\]
In \cite{ls05}, Landi and Suijlekom consider $C(S^7_\theta)$, 
with $\theta_{12}=\theta_{34}=1$ and $\theta_{14}=\theta_{23}=\theta_{13}=\theta_{24}$,
together with an action of $SU(2)$ given by
\begin{align*}
\delta(Z_1)=Z_1\otimes\alpha -Z_2\otimes\alpha^*,\qquad \delta(Z_2)=Z_1\otimes\gamma+Z_2\otimes\gamma^*,\\
\delta(Z_3)=Z_3\otimes\alpha -Z_4\otimes\alpha^*,\qquad \delta(Z_4)=Z_3\otimes\gamma+Z_4\otimes\gamma^*.
\end{align*}
Note that this coaction is well defined because $Z_1$ commutes with $Z_2$ and
$Z_3$ commutes with $Z_4$.

We define two $SU(2)$-equivariant *-homomorphisms
\begin{align*}
\rho_0:C_0((0,1])\otimes C(\mathbb{G})\to C(S^7_\theta): 
\sqrt{\rm t}\otimes\alpha\mapsto Z_1,\;\sqrt{\rm t}\otimes\gamma\mapsto Z_2,\\
\rho_1:C_0((0,1])\otimes C(\mathbb{G})\to C(S^7_\theta): 
\sqrt{\rm t}\otimes\alpha\mapsto Z_3,\;\sqrt{\rm t}\otimes\gamma\mapsto Z_4.
\end{align*}
These maps are well defined because all generators are normal.
Observe that
\[
(\rho_0+\rho_1)({\rm t}\otimes 1)=(\rho_0+\rho_1)({\rm t}\otimes(\alpha^*\alpha+\gamma^*\gamma))
=\sum_{i=1}^4Z^*_iZ_i=1.
\]
Hence, $\dim_{\rm LT}^{SU(2)}(C(S^7_\theta))\leq 1$.\hfill$\diamond$
\end{ex}

Next example shows that both dimensions differ in general.
\begin{ex}[$\mathbb{Z}/2\mathbb{Z}$-action on $3\times 3$ complex matrices]
Let us view $M_3(\mathbb{C})$ as a graph C*-algebra of the graph
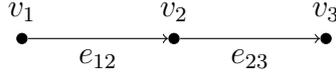
\begin{center}
\begin{tikzpicture}
\node[fill=black,circle,label=$v_1$,inner sep=1.5pt] (0) at (0,0) {};
\node[fill=black,label=$v_2$,circle,inner sep=1.5pt] (1) at (2,0) {};
\node[fill=black,label=$v_3$,circle,inner sep=1.5pt] (2) at (4,0) {};
\path (0) edge[draw,->,below] node {$e_{12}$} (1);
\path (1) edge[draw,->,below] node {$e_{23}$} (2);
\end{tikzpicture}
\end{center}
(e.g., see \cite{BPRS00} for a definition of a graph C*-algebra).
This C*-algebra is generated by three orthogonal projections 
$p_1$, $p_2$, $p_3$, associated to vertices, and two partial isometries $s_{12}$ and $s_{23}$, associated to edges,
satisfying
\[
s_{12}s_{12}^*=p_1,\quad s_{12}^*s_{12}=p_2=s_{23}s_{23}^*,\quad s_{23}^*s_{23}=p_3.
\]
It has the following linear basis:
$p_1$, $p_2$, $p_3$, $s_{12}$, $s_{23}$, $s^*_{12}$, $s^*_{23}$, $s_{12}s_{23}$, $s^*_{23}s^*_{12}$.

We consider a $\mathbb{Z}/2\mathbb{Z}$-action given on the generators by
\[
p_i\mapsto p_i,\quad i=1,2,3,\quad s_{12}\mapsto -s_{12},\quad s_{23}\mapsto -s_{23}.
\]
The above action gives a $\mathbb{Z}/2\mathbb{Z}$-grading and any odd self-adjoint 
element in $M_3(\mathbb{C})$ is of the form
\[
x = k_1 s_{12} + k_2 s_{23} + k^*_1 s^*_{12} + k^*_2 s^*_{23},\qquad k_1, k_2\in\mathbb{C}.
\]
Suppose that we have odd self-adjoint elements $x_j\in M_3(\mathbb{C})$, 
$j=0,1,2,\ldots,n$, such that $\sum_{j=0}^nx_j^2=1.$ Then the relation $p_1+p_2+p_3=1$
(sum of the vertex projections equals 1 for any graph C*-algebra of a graph with finitely many vertices)
leads to a contradiction. Hence we
get that $\dim_{\rm LT}^{\mathbb{Z}/2\mathbb{Z}}(M_3(\mathbb{C}))=\infty$.

It is equally straightforward to find two odd self-adjoint elements of 
$M_3(\mathbb{C})$ whose sum of squares is invertible.
Since there is no $\mathbb{Z}/2\mathbb{Z}$-equivariant $*$-homomorphism
$C(\mathbb{Z}/2\mathbb{Z})\to M_3(\mathbb{C})$,
we conclude that $\dim_{\rm WLT}^{\mathbb{Z}/2\mathbb{Z}}(M_3(\mathbb{C}))=1$.\hfill$\diamond$
\end{ex}

\subsection{The $n$-universal bundles and the strong local-triviality dimension}\label{ltuniversal}

Motivated by \cite[\S~19.2]{St99}, we introduce the following definition.
\begin{df}
We say that a locally trivial compact quantum principal bundle $(A,A^\mathbb{G},\mathbb{G})$
is {\em $n$-universal} if and only if for any other compact quantum principal bundle $(B,B^\mathbb{G},\mathbb{G})$
with $\dim_{\rm LT}^\mathbb{G}(B)\leq n$ there exists \mbox{a~$\mathbb{G}$-equi}\-variant
$*$-homomorphism $A\to B$. 
\end{df}
Recall that, in the case of $\mathbb{Z}/2\mathbb{Z}$-actions, to check if a given 
$\mathbb{Z}/2\mathbb{Z}$-C*-algebra $A$ gives rise to a locally trivial compact quantum principal bundle
one needs to find finitely many odd self-adjoint elements $y_i$ such that $\sum_iy_i^2=1$. This in turn means that there
is a unital $\mathbb{Z}/2\mathbb{Z}$-equivariant $*$-homomorphism $C(S^n_+)\to A$, for some $n$,
given by $x_i\mapsto y_i$.
Hence $C(S^n_+)$ is an $n$-universal compact quantum $\mathbb{Z}/2\mathbb{Z}$-bundle.

In Section~\ref{freejoin}, using the notion of the free noncommutative join, we introduced 
the C*-algebra 
$C(E^{\mathrlap{+}\times}_n\mathbb{G})$ for any compact quantum group $\mathbb{G}$.
We also showed that 
$C(S^n_+)$ and $C(E^{\mathrlap{+}\times}_n\mathbb{Z}/2\mathbb{Z})$ are isomorphic
as $\mathbb{Z}/2\mathbb{Z}$-C*-algebras.
In Theorem \ref{universal} we will prove that 
$C(E^{\mathrlap{+}\times}_n\mathbb{G})$ is an $n$-universal $\mathbb{G}$-bundle
for any $\mathbb{G}$, 
but first we need the following lemma.

\begin{lemma}\label{joinup}
Let $A$ be a $\mathbb{G}$-C*-algebra with a coaction $\delta:A\to A\otimes C(\mathbb{G})$.
Then, $\dim_{\rm LT}^\mathbb{G}(A{\mathrlap{+}\times}  C(\mathbb{G}))\leq \dim_{\rm LT}^\mathbb{G}(A)+1$, where we consider $A{\mathrlap{+}\times}  C(\mathbb{G})$ with the
diagonal action of $\mathbb{G}$.
\end{lemma}
\begin{proof}
Let $\dim_{\rm LT}^\mathbb{G}(A)=n$. We have jointly-unital $\mathbb{G}$-equivariant
$*$-homomorphisms $\rho_i:C_0((0,1])\otimes C(\mathbb{G})\to A$, $i=0,1,\ldots,n$.
Let us define the following $\mathbb{G}$-equivariant $*$-homomorphisms
\[
\widetilde{\rho_i}:C_0((0,1])\otimes C(\mathbb{G})\to A{\mathrlap{+}\times} C(\mathbb{G}),\quad i=0,1,\ldots,n+1,
\]
\[
\widetilde{\rho_i}({\rm t}\otimes h)(s):=(1-s)\rho_i({\rm t}\otimes h),\quad h\in C(\mathbb{G}),\quad i=0,1,\ldots,n,\quad s\in [0,1],
\]
\[
\widetilde{\rho_{n+1}}({\rm t}\otimes h)(s):=s\cdot h,\quad h\in C(\mathbb{G}),\quad s\in [0,1].
\]
It is evident that $\sum_i\widetilde{\rho_i}({\rm t}\otimes 1)=1$, and hence we conclude that 
\[\dim_{\rm LT}^\mathbb{G}(A{\mathrlap{+}\times} C(\mathbb{G}))\leq n+1=\dim_{\rm LT}^\mathbb{G}(A)+1.\]
\end{proof}
\begin{cor}\label{leqn}
Let $n$ be a nonnegative integer and let $\mathbb{G}$ be a compact quantum group with the coproduct 
$\Delta:C(\mathbb{G})\to C(\mathbb{G})\otimes C(\mathbb{G})$. 
Then, $\dim_{\rm LT}^\mathbb{G}(C(E_n^{\mathrlap{+}\times}\mathbb{G}))\leq n$
for the diagonal action of $\mathbb{G}$ on $C(E_n^{\mathrlap{+}\times}\mathbb{G})$.
\end{cor}
\begin{proof}
As $\dim_{\rm LT}(\Delta)=0$, one can proceed by induction using Lemma~\ref{joinup}.
\end{proof}
\begin{ex}[Compact quantum matrix groups]\label{cmqg}
Let $\mathbb{G}$ be a compact quantum matrix group and let the matrix $(u_{jk})\in M_n(C(\mathbb{G}))$
be its fundamental representation (see \cite{Wo87}). Then we can explicitly write down the maps from Definition \ref{lt}
for $C(E^{\mathrlap{+}\times}_d\mathbb{G})$. Indeed, recall that by Theorem \ref{jointhecone3} we have that
\begin{equation}\label{unicone}
C(E^{\mathrlap{+}\times}_d\mathbb{G})\cong (\underbrace{\mathcal{C}C(\mathbb{G})\ast_\mathbb{C}\ldots
\ast_\mathbb{C}\mathcal{C}C(\mathbb{G})}_{d+1})/I,
\end{equation}
where $I$ is the ideal generated by $1-\sum_{i=0}^d {\rm t}_i$\,. Here each ${\rm t}_i$ denotes the inclusion 
of the half-open interval $(0,1]$ into $\mathbb{C}$ in a different copy of $\mathcal{C}C(\mathbb{G})$. Define the following 
$\mathbb{G}$-equivariant $*$-homomorphisms
\[
\rho_i:C_0((0,1])\otimes C(\mathbb{G})\to C(E^{\mathrlap{+}\times}_d\mathbb{G}):
\sqrt{{\rm t}}\otimes u_{jk}\mapsto \left[\sqrt{{\rm t}_i}\otimes u_{jk}\right],
\quad i=0,1,\ldots, d,
\]
where $[\cdot]$ denotes the class of an element of $\mathcal{C}C(\mathbb{G})^{\ast_\mathbb{C}\,d}$ in 
$(\mathcal{C}C(\mathbb{G})^{\ast_\mathbb{C}\,d})/I$.
We slightly abuse notation by denoting the generators
in different copies of the iterated amalgamated free product by the same $u_{jk}$. 
One only needs to check the joint-unitality condition:
\[
\begin{split}
\sum_{i=0}^d\rho_i({\rm t}\otimes 1_{C(\mathbb{G})})&=\sum_{i=0}^d\rho_i\left({\rm t}\otimes \sum_{k=1}^n u^*_{k1}u_{k1}\right)
=\sum_{i=0}^d\sum_{k=1}^n\rho_i\left(\sqrt{\rm t}\otimes u^*_{k1}\right)\rho_i\left(\sqrt{\rm t}\otimes u_{k1}\right)\\
&=\sum_{i=0}^d\sum_{k=1}^n\left[\sqrt{{\rm t}_i}\otimes u^*_{k1}\right]\left[\sqrt{{\rm t}_i}\otimes u_{k1}\right]=
\sum_{i=0}^d\left[{\rm t}_i\otimes \sum_{k=1}^n u^*_{k1} u_{k1}\right]\\
&=\left[\sum_{i=0}^d {\rm t}_i\otimes 1_{C(\mathbb{G})}\right]=\left[1_{C(\mathbb{G})}\otimes 1_{C(\mathbb{G})}\right]
=1_{C(E^{\mathrlap{+}\times}_d\mathbb{G})}.
\end{split}
\]
Here we used the fact that $(u_{jk})$ is a unitary matrix. \hfill$\diamond$
\end{ex}

\begin{thm}\label{universal}
Let $A$ be a $\mathbb{G}$-C*-algebra with a coaction $\delta:A\to A\otimes C(\mathbb{G})$.
Then $\dim_{\rm LT}(\delta)\leq n$ if and only if there exists a $\mathbb{G}$-equivariant unital $*$-homomorphism
$
C(E_n^{\mathrlap{+}\times}\mathbb{G})\to A.
$
In other words, $C(E_n^{\mathrlap{+}\times}\mathbb{G})$ is an $n$-universal
compact quantum principal $\mathbb{G}$-bundle.
\end{thm}
\begin{proof}
First suppose that there is a $\mathbb{G}$-equviariant $*$-homomorphism $C(E^{\mathrlap{+}\times}_n\mathbb{G})\to A$.
By (\ref{loceq}) and Corollary \ref{leqn}, we have that 
$\dim_{\rm LT}^\mathbb{G}(A)\leq \dim_{\rm LT}^\mathbb{G}(C(E^{\mathrlap{+}\times}_n\mathbb{G}))\leq n$.

Now suppose that $\dim_{\rm LT}^\mathbb{G}(A)\leq n$. There exist $\mathbb{G}$-equivariant
$*$-homomorphisms
\[
\rho_i:C_0((0,1])\otimes C(\mathbb{G})\to A,\quad i=0,1,\ldots, n,\quad \text{with}\quad \sum_{i=0}^n\rho_i({\rm t}\otimes 1)=1.
\]
The unitizations $\rho^+_i$, $i=0,1,\ldots,n$, of the above maps (by functoriality of the free product) give rise to
a $\mathbb{G}$-equivariant $*$-homomorphism
\[
\rho:\underbrace{\mathcal{C}C(\mathbb{G})\ast_\mathbb{C}\ldots
\ast_\mathbb{C}\mathcal{C}C(\mathbb{G})}_{n+1}\to A.
\]
Using Theorem \ref{jointhecone3}, it suffices to check if $\rho$ descends to the quotient by the ideal generated by
the element $\sum_i {\rm t}_i-1$. This is however a consequence of the joint-unitality condition.
\end{proof}

Theorem \ref{universal} motivates another definition, where the free noncommutative join 
is replaced by the equivariant noncommutative join.
\begin{df}
Let $A$ a unital $\mathbb{G}$-C*-algebra with a coaction $\delta:A\to A\otimes C(\mathbb{G})$.
Given a nonnegative integer $d$, we say that $\delta$ has the {\em strong local-triviality dimension}
at most $d$, written $\dim_{\rm SLT}(\delta)\leq d$, if there exists a $\mathbb{G}$-equivariant~$*$-homomorphism
$$\rho :C(E_d^{\Delta}\mathbb{G})\longrightarrow A.$$
We set $\dim_{\rm{SLT}}(\delta)=\infty$ if no such $d$ exists.
\end{df}

Since $C(E_0^\Delta\mathbb{G})=C(\mathbb{G})$ and due to Proposition~\ref{dimzero}, for any
coaction $\delta$, we obtain
\begin{equation}
\dim_{\rm WLT}(\delta)=0\iff \dim_{\rm LT}(\delta)=0\iff \dim_{\rm SLT}(\delta)=0.
\end{equation}

Using the same formulas as in the proof of Lemma~\ref{joinup}, 
one can verify the following result (see~\cite[Proposition~3.4]{cp-19}).

\begin{prop}
Let $A$ be a $\mathbb{G}$-C*-algebra with a coaction $\delta:A\to A\otimes C(\mathbb{G})$.
Then, $\dim_{\rm LT}^\mathbb{G}(A\overset{\delta}{\circledast}  C(\mathbb{G}))\leq \dim_{\rm LT}^\mathbb{G}(A)+1$, where we 
consider $A\overset{\delta}{\circledast}  C(\mathbb{G})$ with the analog of the diagonal action of $\mathbb{G}$.
Consequently, $\dim_{\rm LT}^\mathbb{G}(C(E^\Delta_n\mathbb{G}))\leq n$ for all $n$.
\end{prop}

Using the above proposition and the inequality~\eqref{loceq}, 
for any coaction $\delta$, we obtain inequalities between all three local-triviality dimensions
\begin{equation}
\dim_{\rm WLT}(\delta)\leq \dim_{\rm LT}(\delta)\leq \dim_{\rm SLT}(\delta).
\end{equation}

Next example shows that the strong local-triviality dimension differs from the previously defined local-triviality dimensions.
\begin{ex}[Antipodal action on odd free spheres]
Consider the $\mathbb{Z}/2\mathbb{Z}$-action on $C(S^1_+)$ described in
Example~\ref{antifreesphere}.
We will show that $\dim_{\rm SLT}^{\mathbb{Z}/2\mathbb{Z}}(C(S^1_+))\geq 2$.
Indeed, if we would have that $\dim_{\rm SLT}^{\mathbb{Z}/2\mathbb{Z}}(C(S^1_+))\leq 1$,
then we would obtain a $\mathbb{Z}/2\mathbb{Z}$-equivariant $*$-homomorphism
\[
f:C(S^1)\cong C(E_1\mathbb{Z}/2\mathbb{Z})\cong
 C(E^\Delta_1\mathbb{Z}/2\mathbb{Z})\longrightarrow C(S^1_+)
\longrightarrow C(S^1),
\]
where the second arrow is simply the abelianization.
Next, since $K_1(C(S^1_+))=0$, the induced map
\[
f_*:K_1(C(S^1))\longrightarrow K_1(C(S^1))
\]
is trivial. This contradicts the classical Borsuk--Ulam theorem
(see \cite[Theorem~3.1.1]{b-p16phd}).
Similarly, if $n$ is odd, we have that
$\dim_{\rm SLT}^{\mathbb{Z}/2\mathbb{Z}}(C(S^n_+))\geq n+1$. \hfill $\diamond$
\end{ex}

\section{Locally trivial compact principal bundles}\label{classical}

In this section we prove that for unital commutative C*-algebras our definition of the locally trivial 
compact quantum principal bundle recovers
the notion of the locally trivial compact principal bundle in topology.
First, let us recall the definition of a locally trivial principal $G$-bundle, which is the main motivation of this work.

\begin{df}
Let $\pi:X\to M$ be a fiber bundle of topological spaces and let $X$ be equipped
with a right action of a topological group $G$. The quadruple $(X,M,\pi,G)$ is called a {\em principal
$G$-bundle} if the following axioms hold:
\begin{enumerate}
\item For any $x\in X$ and $g\in G$, we have $\pi(xg)=\pi(x)$.
\item For each $m\in M$, there exists an open neighbourhood $U$ of $m$ in $M$ and a fiber-preserving $G$-equivariant
homeomorphism $\varphi:\pi^{-1}(U)\to U\times G$, which is called a~trivialization of $\pi$ over $U$ with a typical
fibre $G$.
\end{enumerate}
\end{df}

Point (2) above describes the local triviality of a principal $G$-bundle. 
From (1) and (2) one can prove that the
action of $G$ on $X$ is free. On the other hand, by a result of Mostow \cite[Theorem~3.1]{mst57}, 
free actions of compact Lie groups on regular topological spaces give rise to locally trivial principal bundles.

We gave a standard definition of a principal G-bundle 
that can be found in textbooks, e.g. \cite{St99, tD08}, and should be contrasted
with the definition of a {\em Cartan principal $G$-bundle} \cite{BHMS07, Ca67}, where one assumes that the action of $G$ is free
and proper and that $M\cong X/G$ (in the original work of Cartan, it is phrased equivalently as the continuity of the translation 
map and a certain density condition) instead of (2).

For our purposes, we need a slight reformulation of the notion of local triviality. This can be achieved
by means of certain invariants that are always finite for locally trivial compact principal $G$-bundles.

\begin{df}[\cite{Sch61}]
The {\em Schwarz genus}~ of a $G$-space $X$, denoted by ${\rm g}_G(X)$, is the smallest number $n$ such that
$X$ can be covered with open $G$-invariant subsets $U_0,U_1,\ldots,U_n$ with the property that
for every $i=0,1,2,\ldots,n$, there exists a $G$-equivariant map $U_i\to G$. If no such $n$ exists,
we write ${\rm g}_G(X)=\infty$.
\end{df}

Now let $X$ and $Y$ be two $G$-spaces and suppose that there exists a $G$-map $X\to Y$.
Then, there is an inequality
\begin{equation}\label{scheqi}
{\rm g}_G(X)\leq {\rm g}_G(Y).
\end{equation}

Note that if $G$ is a compact Hausdorff group acting on a compact Hausdorff space $X$, then
${\rm g}_G(X) <\infty$ if and only if $\pi:X\to X/G$ is a locally trivial principal $G$-bundle. One can say even more,
if ${\rm g}_G(X)=n$, this means that $X/G$ can be covered with at most $n+1$ trivializing open sets. Recall that
\begin{equation}\label{schdim}
{\rm g}_G(X) \leq \dim(X/G). 
\end{equation}
Indeed, $\dim (X/G) \leq d$ if and only if any finite open cover of $X/G$ can be refined by a finite open cover that splits into $(d+1)$ disjoint families of open sets. Note that open subsets of a trivializing open set are still trivializing, and thus if a finite open cover of $X/G$ consists of trivializing open sets, then any finite open cover that refines it also consists of trivializing open sets. Combining this with the observation that a disjoint union of trivializing open sets is still a trivializing open set, we obtain the desired inequality. 


In \cite[Theorem~9]{Sch61} Schwarz showed that for any topological group $G$ there is an inequality
\begin{equation}\label{schmil}
{\rm g}_G(E_nG)\leq n.
\end{equation}

We need another invariant which was introduced for purposes of the Borsuk--Ulam-type theorems 
(e.g., see \cite{Ma03}).

\begin{df}
Let $X$ be a $G$-space. We define the {\em $G$-index} of $X$ by
\[
{\rm ind}_G(X):={\rm min}\{n:\exists~G\text{-}{\rm map}~X\to E_nG\}.
\]
If there is no such $G$-map, we write ${\rm ind}_G(X)=\infty$.
\end{df}

As for the Schwarz genus, we have that if there is an $G$-equivariant map $X\to Y$ between two
$G$-spaces, then 
\begin{equation}\label{indeqi}
{\rm ind}_G(X)\leq {\rm ind}_G(Y).
\end{equation}
The following result shows that the above invariants are equal for compact Hausdorff spaces.
As we did not find it in the literature, we give its proof for reader's convenience.
\begin{prop}\label{schind}
Let $G$ be a compact Hausdorff group acting continuously on a compact Hausdorff $G$-space $X$. Then,
\[
{\rm g}_G(X)={\rm ind}_G(X).
\]
\end{prop}
\begin{proof}
First assume that ${\rm ind}_G(X)=n$. We have a $G$-map $X\to E_nG$. 
Using \eqref{scheqi} and \eqref{schmil}, 
we get that ${\rm g}_G(X)\leq {\rm g}_G(E_nG)\leq n={\rm ind}_G(X)$.

Now assume that ${\rm g}_G(X)=n$. 
Then, we know that there are $G$-equivariant maps $\phi_i \colon U_i=\pi^{-1}(V_i) \to G$ for some
trivializing cover $\{V_i\}_{i=0}^{n}$. 
Let $\{f_i\}_{i=0}^{n}$ be a partition of unity subordinate to $\{V_i\}_{i=0}^{n}$. We define a $G$-map 
\[
    \psi \colon X \to E_n G \; , \quad x \mapsto \sum_{i=0}^{n} f(\pi (x)) \phi_i(x) \; ,
\]
where we use the simplicial notation for the multi-join. 
The above $G$-map is well defined. Indeed, the condition $\sum_{i=0}^nf_i(m)=1$,
for every $m\in X/G$, assures that the image of $\psi$ lands in $\Delta^{n}\times G^{n+1}$, where $\Delta^{n}$ is the $n$-simplex.
Equations \eqref{indeqi} and \eqref{schmil}, and the existence of $\psi$  imply that 
${\rm ind}_G(X)\leq {\rm ind}_G(E_nG)\leq n={\rm g}_G(X)$.
\end{proof}
We emphasize that Proposition \ref{schind} shows that, for compact Hausdorff $X$ and $G$, ${\rm ind}_G(X)<\infty$ 
if and only if the $G$-bundle $\pi:X\to X/G$ is locally trivial. Thus, we obtain a different characterization of 
locally trivial compact principal $G$-bundles.

Next, we show that for unital commutative C*-algebras, i.e. algebras of complex-valued continuous functions on compact Hausdorff topological spaces, Definition \ref{lt}
recovers the usual notion of local triviality.
\begin{thm}
Let $G$ be a compact Hausdorff group, let $X$ be a compact Hausdorff space, and 
let $G$ act continuously on $X$. Denote by
$\alpha\colon G\to{\rm Aut}(C(X))$ the induced action.
Then, 
\begin{equation}
{\rm ind}_G(X)={\rm dim}_{\rm LT}(\alpha).
\end{equation}
\end{thm}
\begin{proof}
Let ${\rm dim}_{\rm LT}(\alpha)=n$. 
Take equivariant $*$-homomorphisms from Definition~\ref{lt} and denote by
\[
\rho_i:(C_0((0,1])\otimes C(G))^{+}\cong C(\mathcal{C}G)\rightarrow C(X),\qquad i=0,1,\ldots,n,
\]
their unitizations. 
Note that $\sum_{i=0}^n\rho_i({\rm t}\otimes 1)=1$. We dualize the above to obtain $G$-equivariant continuous maps
\[
\hat{\rho}_i:X\rightarrow\mathcal{C}G, \qquad i=0,1,\ldots,n.
\]
Next, we define a continuous $G$-map
\[
\psi:X\rightarrow E_nG,\quad
x\mapsto \sum_{i=0}^n\hat{\rho}_i(x),
\]
where we used simplicial coordinates for $E_nG$. Every element $\hat{\rho}_i(x)$ is a class in $[0,1]\times G$. Let us fix an 
arbitrary point $x\in X$ and let $\hat{\rho}_i(x):=[(s_i,g_i)]$ for some $s_i\in [0,1]$ and $g_i\in G$. We only need to check if $
\sum_{i=0}^n s_i=1$ (with each $s_i$ being non-zero):
\[
1=\sum_{i=0}^{n}(\rho_i({\rm t}\otimes 1))(x)
=\sum_{i=0}^{n}({\rm t}\otimes 1)(\hat{\rho}_i(x))
=\sum_{i=0}^{n}({\rm t}\otimes 1)([(s_i,g_i)])
=\sum_{i=0}^{n}s_i\,.
\]
Hence, ${\rm ind}_G(X)\leq n={\rm dim}_{\rm LT}(\alpha)$.

Now suppose that ${\rm ind}_G(X)=n$. We have a $G$-equivariant map
$
\psi:X\rightarrow E_nG.
$
Define $G$-equivariant continuous maps
$$
\hat{\rho}_i:X\overset{\psi}{\rightarrow}E_nG\overset{\rm \iota}{\hookrightarrow}(\mathcal{C}G)^{n+1}\overset{pr_{i+1}}
{\twoheadrightarrow}\mathcal{C}G~.
$$
where $i=0,1,\ldots,n$. Again each element $\hat{\rho}_i(x)$ is a class in $[0,1]\times G$. Let us denote it as previously by 
$\hat{\rho}_i(x):=[(s_i,g_i)]$, where $x\in X$ is arbitrary. By definition $\sum_{i=0}^{n} s_i=1$ and using the above calculation 
one can show that $\sum_{i=0}^{n}\rho_i({\rm t}\otimes 1)=1$. This implies that 
${\rm dim}_{\rm LT}(\alpha)\leq n={\rm ind}_G(X)$.
\end{proof}

\begin{rem}
Observe that for any compact Hausdorff group $G$ and any compact Hausdorff $G$-space $X$
we obtain
\[
\dim_{\rm WLT}^G(C(X))=\dim_{\rm LT}^G(C(X))=\dim_{\rm SLT}^G(C(X)).
\] 
However, in the previous section we showed that all the dimensions can take different values in general.\hfill$\diamond$
\end{rem}

We end this section by some remarks of C*-algebraic flavour.
Let $G$ be a compact Hausdorff group, $X$ a compact Hausdorff space, 
and $\pi:X\to X/G$ be a compact principal $G$-bundle with the local-triviality dimension equal to $n$.
Then we have $G$-equivariant $*$-homo\-morphisms $\rho_i:C_0((0,1])\otimes C(G)\to C(X)$,
for $i=0,\ldots,n$, such that $\sum_i\rho_i({\rm t}\otimes 1)=1$.
Now let $p_i:=\rho_i({\rm t}\otimes 1)\in C(X)$. Each $p_i$ is a positive contractive element.
We consider C*-subalgebras of $C(X)$ of the form
\[
A_i:=\{\sqrt{p_i}f\sqrt{p_i}~|~f\in C(X/G)\}^{\rm cls}\subseteq C(X/G),\quad i=0,...,n.
\]

Since $C(X)$ is commutative, each $A_i$ is an ideal of
functions supported on an open set $U_i\subseteq X/G$ over which the principal bundle $X\to X/G$
is trivial.
By the joint-unitality condition, we have that
$$
1\in A_0+A_1+\ldots+A_n.
$$
Next, each $\rho_i$ induces a $\mathbb{G}$-equivariant unital $*$-homomorphism
$$
\psi_i:C(G)\to \mathcal{M}(A_i),
$$
where $\mathcal{M}(A_i)$ is the multiplier algebra of $A_i$. 

\begin{rem}
Note that one could consider the hereditary C*-subalgebras $A_i$ for actions of compact quantum groups
on arbitrary unital C*-algebras with finite local-triviality dimension. However, further investigations are needed
to establish if they play a similar role to ideals of functions supported on the open sets constituting
a~trivializing cover.\hfill$\diamond$
\end{rem}

\section{Relations with piecewise triviality and the Rokhlin dimension}\label{relations}

In this section, we examine the connection between the local-triviality dimension and some other related notions, i.e. piecewise triviality and the Rokhlin dimension.

\subsection{Piecewise triviality}

We start by exploring the connection of the local-triviality dimension with piecewise triviality. 
We recall the definition of piecewise triviality in the classical context.

\begin{df}
A Cartan principal bundle $(X,\pi, M, G)$ is called {\em piecewise trivial}, 
if there exist a covering of $M$ by finitely many closed
sets $V_0,\ldots, V_n$ and fibre-preserving $G$-equivariant homomorphisms $\chi_i:\pi^{-1}(V_i)\to V_i\times G$,
$i=0,1,\ldots,n$.
\end{df}

This definition was 
introduced in \cite{BHMS07} along with an example (the {\em bubble space}) of a Cartan principal $G$-bundle 
that is piecewise trivial, but not locally trivial.
For compact Hausdorff spaces, local triviality implies piecewise triviality. Indeed, for the cover $\{V_i\}$ in the above
definition, take the supports of functions of a partition of unity subordinate to the open trivializing cover given by local triviality.

The concept of piecewise triviality admits a straightforward generalization to the realm of noncommutative geometry.

\begin{df}
Let $A$ be a unital \mbox{$\mathbb{G}$-C*-al}ge\-bra, where $\mathbb{G}$ is a compact quantum group. An action of
$\mathbb{G}$ on $A$ is said to be {\em piecewise trivial}~\cite{HKMZ11}, if for some $n\in\mathbb{N}$ there exist $\mathbb{G}$-invariant closed ideals
$I_0,\ldots,I_n$ of $A$, such that $\bigcap_{i=0}^nI_i=0$, and unital $\mathbb{G}$-equivariant $*$-homomorphisms
$\chi_i:C(\mathbb{G})\to A/I_i$.
\end{df}
Note that according to this definition, if a simple $\mathbb{G}$-C*-algebra is piecewise trivial
it is in fact trivializable.

In contrast with the classical case, local triviality in the sense of Definition \ref{lt}
does not imply piecewise triviality, as the latter notion requires the existence of proper ideals, while the former may be even applied to simple algebras, as the next example shows. 

\begin{ex}\label{nct}(A $\mathbb{Z}/2\mathbb{Z}$-action on the irrational rotation algebra).
Let $\theta\in(0,1)$ be an irrational number. The {\em irrational rotation algebra} $A_\theta$ (or the {\em noncommutative torus}; 
see \cite{EfHa67, PiVo80, Rie81}) is the universal C*-algebra generated by two unitaries $U$ and $V$ subject to the relation
\[
	UV=e^{2\pi i\theta}VU.
\]
This simple C*-algebra plays a fundamental role in noncommutative geometry.
We define an involutive automorphism of $A_\theta$ by mapping 
\[
	U\mapsto -U \text{ and } V\mapsto V.
\]
This gives us an action of $\mathbb{Z}/2\mathbb{Z}$ on $A_\theta$. Note that the subalgebra $C^*(U)$ generated by $U$ is isomorphic to $C(S^1)$ and invariant under the above action, and that the restricted action amounts to the antipodal action on $S^1$, whose local triviality dimension is $1$. Hence, applying  inequality \eqref{loceq} to the equivariant embedding $C^*(U) \subset A_\theta$, we see that $\dim_{\rm LT}^{\mathbb{Z}/2\mathbb{Z}}(A_\theta)\leq 1$.

However, this action cannot be piecewise trivial,
because simplicity of $A_\theta$ would force it to be trivial, which we will show is not possible.

We show that the considered bundle is not trivial and thus $\dim_{\rm LT}^{\mathbb{Z}/2\mathbb{Z}}(A_\theta)=1$. 
Indeed, if this were not the case, we would have a unital $\mathbb{Z}/2\mathbb{Z}$-equivariant \mbox{$*$-homo}\-morphism 
$\varphi \colon 
C(\mathbb{Z}/2\mathbb{Z})\to A_\theta$. Let $p$ and $q$ be the images under $\varphi$ of the two minimal projections generating $C(\mathbb{Z}/2\mathbb{Z})$. Thus $p$ and $q$ are orthogonal projections that add up to $1$ and they are 
translates of each other under the $\mathbb{Z}/2\mathbb{Z}$ action. 
The unique trace $\tau \colon A_\theta \to \mathbb{C}$ given by 
\[
	\tau \left(\sum_{m , n \in \mathbb{Z}} a_{m,n} U^m V^n \right) = a_{0,0}
\]
is clearly invariant under $\mathbb{Z}/2\mathbb{Z}$-action on $A_\theta$. We compute
\[
	\tau(p) = \tau(q) = \frac{\tau(p+q)}{2} = \frac{1}{2} \; .
\]
This is impossible since we know the image of the homomorphism $K_0(A_\theta) \to \mathbb{R}$ induced by $\tau$ is $\mathbb{Z}+\theta\mathbb{Z}$ (this is in fact also injective; see \cite{PiVo80}).

Alternatively, we can prove this fact by noticing that the involution $U\mapsto -U$ and $V \mapsto V$ is homotopic to the identity via the homotopy defined by $U \mapsto e^{\pi i t} U$ and $V \mapsto V$. This implies that $p$ and $q$ induce the same element in $K_0(A_\theta)$ and $2 [p] = [1]$, which contradicts with the fact that $[1]$ is not divisible by $2$ in $K_0(A_\theta)$. 

We have proved that the noncommutative bundle under consideration is not trivial. Thus we conclude that
$\dim_{\rm LT}^{\mathbb{Z}/2\mathbb{Z}}(A_\theta)=1$. \hfill$\diamond$
\end{ex}
\subsection{Rokhlin dimension}\label{rokdimension}
We proceed to the relation of the local-triviality dimension with the Rokhlin dimension. 
Throughout the subsection all C*-algebras are assumed to be separable and groups are assumed to be metrizable
unless otherwise stated.
Let us start by recalling the definitions of a sequence 
algebra and a completely positive contractive order zero map, 
which are the basic ingredients of the definition of the Rokhlin dimension.

\begin{df}
Let $A$ be a separable unital C*-algebra, $\ell^\infty(\mathbb{N},A)$ denote the C*-algebra of all bounded sequences with 
elements in $A$ and $c_0(\mathbb{N},A)$
denote the ideal consisting of sequences converging to zero in norm. 
The \emph{sequence algebra} is defined as the quotient
\[
	A_\infty:=\ell^\infty(\mathbb{N},A)/c_0(\mathbb{N},A).
\]
The \emph{central sequence algebra} is defined as the commutant $A_\infty\cap A'$.
\end{df}

If $G$ is a compact metrizable group acting on a separable unital C*-algebra $A$,
then there are actions of $G$ on both $A_\infty$ and $A_\infty\cap A'$.
The continuity of those actions is a consequence of a result of Brown \cite[Theorem~2]{l-b00}.

\begin{df}[\cite{WZ09}]
Let $A$ and $B$ be C*-algebras. A completely positive contractive map $\varphi:A\to B$ is called {\em order zero} if
and only if
\[
\text{$\varphi(a)\varphi(b)=0$, whenever $ab=0$, for any $a,b\in A$.}
\]
\end{df}

The theory of completely positive contractive order zero maps was developed by Winter and Zacharias 
and has played
a fundamental role in the recent breakthrough in the classification theory of C*-algebras.
The following result, 
based on the Stinespring theorem, will be crucial in exploring the connection of the Rokhlin
dimension with local triviality.

\begin{thm}[\cite{WZ09}]\label{winzach} 
Let $A$ and $B$ be unital C*-algebras.
Any completely positive contractive order zero map $\varphi: A\to B$ induces the 
$*$-ho\-mo\-mor\-phism
\[
\text{$\rho_\varphi:C_0((0,1])\otimes A\to B$ determined by $\rho_\varphi({\rm t}\otimes a):=\varphi(a)$ for all $a\in A$.} 
\]
Conversely, any $*$-homomorphism $\rho:C_0((0,1])\otimes A\to B$ induces the completely positive order zero map 
\[
\text{$\varphi_\rho:A\to B$ given by $\varphi_\rho(a):=\rho({\rm t}\otimes a)$ for all $a\in A$.}
\]
\end{thm}
\noindent
The above is also true in the equivariant setting for actions of locally compact group on C*-algebras 
(see \cite[Corollary~2.10]{Ga17}). 
Note that to discuss order zero maps we need not restrict to separable C*-algebras.


Definition \ref{lt} was inspired by the following definition.
\begin{df}[\cite{Ga17}]\label{rd}
Let $G$ be a compact metrizable group and let $\delta:G\to Aut(A)$ be an action of $G$ on a separable unital C*-algebra $A$.
We say that an action $\delta$
has the \emph{Rokhlin dimension} $n$, written $dim_{\rm Rok}(\delta)=n$, if $n$ is the minimal non-negative
integer such that there exist $G$-equivariant completely positive contractive order zero maps 
\[
\varphi_0,\ldots,\varphi_n:C(G)\to A_\infty\cap A'\qquad \text{with}\quad \sum_{i=0}^{n} \varphi_i(1) = 1.
\]
We set $\dim_{\rm Rok}(\delta)=\infty$ if no such $n$ exists.
\end{df}

Using Theorem~\ref{winzach}, one can compare the local-triviality dimension and the Rokhlin dimension.
Note that the local-triviality dimension is defined for actions of compact quantum groups to begin with,
while the generalization of the Rokhlin dimension
to actions of compact quantum groups is not straightforward and requires some reformulations~\cite{GKL17}.

\begin{rem}
Observe the original definition of the Rokhlin dimension
\cite[Definition~3.2]{Ga17} is slightly more general: $G$ is assumed to be compact and second countable,
while $A$ is assumed to be $\sigma$-unital. Furthermore, the (corrected) relative central sequence algebra is used instead of 
$A_\infty\cap A'$. Note however that we do not use the invariant part $A_{\infty,\delta}$ of the sequence algebra due to 
\cite[Theorem~2]{l-b00}.\hfill$\diamond$
\end{rem}

Next, using the notion of equivariant projectivity, we show that, for compact Lie group actions
on unital commutative separable C*-algebras, the notions of the local-triviality dimension and the Rokhlin dimension coincide. 

Let us first state the definition of the equivariant projectivity in the case of compact Hausdorff group actions.
\begin{df}[\cite{Ph12, PST15}]\label{proj}
Let $G$ be a compact Hausdorff group and let $A$ be a $G$-C*-algebra. We say that $A$ is \emph{$G$-equivariantly projective} 
if for any $G$-C*-algebra $B$, a $G$-invariant closed ideal $J\subseteq B$, and an equivariant $*$-homomorphism $\sigma:A\to 
B/J$, there is an
equivariant $*$-homomorphism $\lambda: A\to B$ such that
$\pi\circ\lambda=\sigma$, where $\pi:B\to B/J$ is the quotient map.
\end{df}
Note that the above definition means projectivity in the category of general \mbox{$G$-C*-algebras} and one can
restrict this definition to subcategories of unital $G$-C*-algebras, commutative $G$-C*-algebras, etc.
In the category of commutative C*-algebras, 
an object $C_0(X)$ is $G$-equivariantly projective if and only if
$X$ is a $G$-AR ($G$-equivariant absolute retract \cite{Mu82, Mu83}).

\begin{rem}
The notion of equivariant projectivity of \cite{Ph12,PST15} should be distinguished from 
the notion of equivariant projectivity in the category of projective modules that can be applied to C*-algebras \cite{BH09}.
\hfill$\diamond$
\end{rem}

We state two propositions that establish a relation between the local-triviality dimension and the Rokhlin dimension.

\begin{prop}
Let $A$ be a unital commutative separable C*-algebra equipped with an action $\delta$ of a compact metrizable group $G$.
Then, \[\dim_{\rm Rok}(\delta)\leq \dim_{\rm LT}(\delta).\]
\end{prop}

\begin{proof}\label{rokleqlt}
Assume that $\dim_{\rm LT}(\delta)=d$. Due to Theorem \ref{winzach}, we have $G$-equivariant completely positive
contractive order zero maps 
$\varphi_0,\ldots,\varphi_d:C(G)\to A$ such that $\sum_{i=0}^d\varphi_i(1)=1$. 
Using the unital inclusion $\iota:A\to A_{\infty}$ and the fact that $A$ is commutative, we obtain $G$-equivariant completely 
positive contractive order zero maps
\[
\widetilde{\varphi}_i:=\iota\circ\varphi_i: C(G)\to A_{\infty}=A_{\infty}\cap A',\quad i=0,1,2,\ldots,d.
\]
Since $\iota$ is unital, we obtain $\sum_{i=0}^d\widetilde{\varphi}_i(1)=1$. Therefore, 
$\dim_{\rm Rok}(\delta)\leq d=\dim_{\rm LT}(\delta)$.
\end{proof}

\begin{prop}\label{ltleqrok}
Let $G$ be a compact metrizable group such that $C_0((0,1])\otimes C(G)$ is $G$-equivariantly projective
and let $\delta$ be an action of $G$ on a unital separable C*-algebra $A$.
Then, \[\dim_{\rm WLT}(\delta)\leq \dim_{\rm Rok}(\delta).\]
\end{prop}
\begin{proof}
Suppose that $\dim_{\rm Rok}(\delta)=d$ and we have equivariant completely 
positive contractive order zero maps
\[
	\varphi_i:C(G)\to A_{\infty}\cap A'\quad \text{for}\quad i=0,1,\ldots,d.
\] 
Using Theorem \ref{winzach}, we obtain equivariant $*$-homomorphisms 
\[
	\rho_i:C_0((0,1])\otimes C(G)\to A_{\infty}\cap A'\hookrightarrow A_{\infty}\,.
\]

Since $C_0((0,1])\otimes C(G)$ is $G$-equivariantly projective,
for each $\rho_i$, there exists a~$G$-equivariant $*$-homomorphism
\[
	\lambda_i: C_0((0,1])\otimes C(G)\to \ell^\infty(\mathbb{N},A)
\]
such that $\pi\circ\lambda_i=\rho_i$, where $\pi:\ell^\infty(\mathbb{N},A)\to A_\infty$ is the quotient map.
Define $pr_n:\ell^\infty(\mathbb{N},A)\to A$ as a projection on the $n$th element of the sequence for some $n\in\mathbb{N}$.
Then, the maps 
\[
	\widetilde{\rho}_{n,i}:=pr_n\circ\lambda_i:C_0((0,1])\otimes C(G)\to A,\quad i=0,1,2,\ldots,d,
\] 
define $G$-equivariant $*$-homomorphisms.

From the fact that $\sum_i\rho_i({\rm t}\otimes 1)=1$
and that $\rho_i=\pi\circ\lambda_i$, we obtain
\[
	\left\|pr_n\left(\sum_{i=0}^d\lambda_i({\rm t}\otimes 1)\right)-1\right\|\to 0\quad\rm{as}\quad n\to\infty.
\]
Hence for large enough $N$, we can conclude that
\[
	\sum_{i=0}^d\widetilde{\rho}_{N,i}({\rm t}\otimes 1)=pr_N \left( \sum_{i=0}^d \lambda_i({\rm t} \otimes 1) \right)
\] 
is invertible. 
\end{proof}

If $G$ is a compact Lie group, then the space $(0,1]\times G$ is a $G$-AR (see \cite[Corollary~2.3]{An02}). 
Hence, $C_0((0,1])\otimes C(G)$ is $G$-equivariantly projective in the category of commutative
C*-algebras.
In the general possibly noncommutative case, it is known that $C_0((0,1])\otimes C(\mathbb{Z}/2^n\mathbb{Z})$
is equivariantly projective for any $n\in\mathbb{N}\setminus\{0\}$ \cite[Proposition~2.10]{PST15}.

Combining Propositions~\ref{rokleqlt} and~\ref{ltleqrok}, and taking advantage of the above remark about Lie groups, 
we arrive at
\begin{thm}\label{roktriv}
Let $A$ be a unital commutative separable C*-algebra equipped with an action $\delta$ of a compact Lie group $G$.
Then, \[\dim_{\rm LT}(\delta) = \dim_{\rm Rok}(\delta).\]
\end{thm}

Theorem~\ref{roktriv} suggests that the Rokhlin dimension can be viewed as another noncommutative generalization of local triviality of compact
principal $G$-bundles, where $G$ is a compact Lie group. However, when we stay away from Lie groups, these notions differ 
even in the classical case as shown in the Theorem \ref{dimdrop} below.

First, we need to introduce a new characterization of the Rokhlin dimension for actions of compact metrizable
groups on unital separable C*-algebras. 
	To this end, let us recall that every compact metrizable group $G$ contains a decreasing sequence of normal subgroups
	\[
	G = N_1 \supset N_2 \supset \ldots \text{ with }\bigcap\{N_i:i=1,2,\ldots\}=\{e\} \; ,
	\]
	such that, for every $i$,
	$H_i:=G/N_i$ is a compact Lie group. Thus $G=\varprojlim H_i$. (See for example \cite[Theorem~53]{Pon46}).
	
	Let $\delta$ denote the action of a compact metrizable group $G$ 
	on a unital C*-algebra $A$. For any normal subgroup $N_i\subseteq G$ as
	above, we can define the action $\delta^i$ of $H_i$ on $A^{N_i}$ by the formula
	\[
	\delta^i_{[g]_i}(a):=\delta_g(\iota_i(a)),\qquad g\in G,\quad [g]_i\in H_i, \quad a\in A^{N_i}.
	\]
	Here $\iota_i$ is the inclusion of $A^{N_i}$ into $A$.

\begin{thm}\label{roknew}
Let $A$ be a unital separable C*-algebra equipped with an action $\delta$ of a compact metrizable group 
$G=\varprojlim H_i=\varprojlim (G/N_i)$ and let $\delta^i$ denote the action of $H_i$ on $A^{N_i}$.
Then,
\[
	\dim_{\rm Rok}(\delta)=\sup_{i}\left\{\dim_{\rm Rok}(\delta^i)\right\}\;.
\]
\end{thm}
\begin{proof}
First note that $A=\varinjlim A^{N_i}$ and there is an action of $G$ on each $A^{N_i}$
defined by the quotient map $G\to H_i$.
Then, using \cite[Theorem~3.8 (4)]{Ga17}, we obtain
\[
	\dim_{\rm Rok}(\delta)\leq \liminf_i\left\{\dim_{\rm Rok}(\delta^i)\right\}\leq\sup_i\left\{\dim_{\rm Rok}(\delta^i)\right\}.
\] 

Next, suppose that $\dim_{\rm Rok}(\alpha)=d$, namely that there are jointly-unital 
\mbox{$G$-equi}\-variant completely positive contractive order zero maps
\[
	\varphi_j:C(G)\longrightarrow A_\infty\cap A',\qquad j=0,1,\ldots, d.
\] 
For every $i$, there is a $G$-equivariant $*$-homomorphism
\[
	C(H_i)\longrightarrow C(G),
\]
where the action of $G$ on $H_i$ is again defined by the quotient map $G\to G/N_i$.
Composing the two maps above, for any $i$\,, we get jointly unital
$G$-equivariant completely positive contractive order zero maps
\[
	\psi^i_j:C(H_i)\longrightarrow A_\infty\cap A',\qquad j=0,1,\ldots,d.
\]
Note that, as a $G$-space, $H_i$ is $N_i$-invariant so that the image of the above map is contained in the fixed-point 
subalgebra under the $N_i$-action. Therefore, for all $i$, we obtain jointly-unital $H_i$-equivariant 
completely positive contractive order zero maps
\[
	\psi^i_j:C(H_i)\longrightarrow \left(A_\infty\cap A'\right)^{N_i},\qquad j=0,1,\ldots,d.
\]
Observe that there is a $*$-homomorphism 
\[
	S:\left(A_{\infty}\right)^{N_i} \to \left(A^{N_i}\right)_{\infty}
\]
given by averaging each component over $N_i$\,, i.e.,
\[
	S:[(a_i)_i] \mapsto \left[ \left( \int_{N_i} \alpha_n(a_i) \, d\mu(n) \right)_i\, \right] \; ,
\]
where $\mu$ denotes the normalized Haar measure on the compact group $N_i$. 
Moreover, it is straightforward to check that
\[
	S\left(\left(A_\infty\cap A'\right)^{N_i}\right)\subseteq (A^{N_i})_\infty\cap (A^{N_i})'.
\]
Hence, for every $i$, 
we obtain jointly unital $H_i$-equivariant 
completely positive contractive order zero maps
\[
	\psi^i_j:C(H_i)\longrightarrow \left(A^{N_i}\right)_\infty\cap \left(A^{N_i}\right)',\qquad j=0,1,\ldots,d.
\]
This implies that, for any $i$,
\[
	\dim_{\rm Rok}(\delta^i)\leq d=\dim_{\rm Rok}(\delta).
\]
Subsequently,
\[
	\sup_i\left\{\dim_{\rm Rok}(\delta^i)\right\}\leq \dim_{\rm Rok}(\delta).
\]
\end{proof}
Intuitively speaking, the Rokhlin dimension of an action does not see the small subgroups $N_i$. 
Let $G=\varprojlim G/N_i=\varprojlim H_i$ be a compact metrizable group and let $X$ be a compact metrizable $G$-space.
For the purposes of the next corollary, Theorem~\ref{rokupbd}, and Theorem~\ref{dimdrop}, 
we introduce the notion of the {\em homotopy quotient} $X{\times}_GH_i$\,.
Let us assume that the $G$-action on $X$ is on the right and we define a left action of $G$ on $H_i$ using the quotient map
as before. Then, there is a diagonal $G$ action on $X\times H_i$ and 
\[
X{\times}_GH_i:=(X\times H_i)/G.
\]
We equip $X{\times}_GH_i$ with a right $H_i$-action. Note that there is a isomorphism of $H_i$-spaces
\[
	X\underset{G}{\times}H_i\cong X/N_i\,.
\]
Note however that the homotopy quotient construction is much more functorial than the orbit space construction.

\begin{cor}\label{roknewab}
Let $G=\varprojlim H_i$ be a compact metrizable group and let $X$ be a compact metrizable $G$-space.
Let $\delta$ denote the induced $C(G)$-coaction on $C(X)$.
The following equality holds:
\[
	\dim_{\rm Rok}(\delta)=\sup_{i}\left\{{\rm ind}_{H_i}\left(X\underset{G}{\times}H_i\right)\right\}.
\]
\end{cor}

The following result gives a nice upper bound for the Rokhlin dimension in the commutative case.

\begin{thm}\label{rokupbd}
Let $X$ be a compact metrizable space equipped with an action $\delta$ of a compact metrizable group $G$ and let $\dim(X/G)<\infty$.
Then, $\dim_{\rm Rok}(\delta)<\infty$ if and only if $\delta$ is free. In fact, we have $\dim_{\rm Rok}(\delta) \leq \dim(X/G)$. 
\end{thm}
\begin{proof}
We already know that $\dim_{\rm Rok}(\delta)<\infty$ implies freeness \cite[Theorem~4.1(1)]{Ga17}. 
Conversely, assume that the action of $G$ on $X$ is free.
Then, for any $H_i$, the action of $H_i$ on $X{\times}_GH_i$ is free as well. By~Mostow's theorem
\cite[Theorem~3.1]{mst57}, we know that
any free action of a compact Lie group is locally trivial, so we have ${\rm ind}_{H_i}(X{\times}_GH_i)<\infty$
for any $H_i$. By Proposition~\ref{schind} and the inequality~(\ref{schdim}), we have ${\rm ind}_{H_i}(X{\times}_GH_i) \leq \dim((X{\times}_GH_i) / H_i)$, but the latter space is homeomorphic to $X/G$ for any $H_i$. 
Therefore, Theorem~\ref{roknew} implies that $\dim_{\rm Rok}(\alpha) \leq \dim(X/G) < \infty$.  
\end{proof}





Next, we present the following striking dimension reduction phenomenon. 

\begin{thm}\label{dimdrop}
	For any compact metrizable space $X$ equipped with a continuous action by the $p$-adic group 
	$\mathbb{Z}_p$. Let $\delta$ denote the induced $C(G)$-coaction on $C(X)$.
Then, if ${\dim}_{\rm Rok}(\delta) < \infty$, then ${\dim}_{\rm Rok}(\delta) \leq 3$. 
\end{thm}

\begin{proof}
For the sake of brevity, we write $G:=\mathbb{Z}_p$, $H_i:= \mathbb{Z}/p^i\mathbb{Z}$ and $X_i:=X/N_i$ for all $i$.
One can show that
\[
	X\underset{G}{\times}H_n \cong {\varprojlim_{m \geq n}} \left(X_m\underset{H_m}{\times}H_n\right),
\]
which implies that
\[
	{\rm ind}_{H_n}\left(X\underset{G}{\times}H_n\right)
	\leq\underset{m\geq n}{\min}\left\{{\rm ind}_{H_n}\left(X_m\underset{H_m}{\times}H_n\right)\right\}.
\]

We compute the upper bound for the Rokhlin dimension:
	
\begin{align*}
\dim_{\rm Rok}(\delta)&
=\max_n\left\{{\rm ind}_{H_n}\left(X\underset{G}{\times} H_n\right)\right\}
\leq \max_n\min_{m\geq n}\left\{{\rm ind}_{H_n}\left(X_m\underset{H_m}{\times} H_n\right)\right\}\\
&=\max_n\min_{m\geq n}\min
\left\{k\;\big{|}\;\exists\; X_m\underset{H_m}{\times}H_n\overset{H_n}{\longrightarrow} E_kH_n\right\}\\
&=\max_n\min_{m\geq n}\min
\left\{k\;\big{|}\;\exists\; X_m\overset{\mathbb{Z}/p^m}{\longrightarrow} E_kH_n\right\}.
\end{align*}
Here 
the last step uses the fact that if $H$ is a quotient group of a compact group $G$, $X$ is a $G$-space, and $Y$ is an $H$-space 
(and hence also a $G$-space), then $\exists\; X\overset{G}{\longrightarrow} Y$ if and only if 
$\exists\; X {\times}_G H \overset{H}{\longrightarrow} Y$.

From Theorem~\ref{roktriv}, we know that 
${\rm ind}_{H_m}(X_m)=\dim_{\rm Rok}^{H_m}(X_m)<\infty$,
therefore 
$\exists\; X_m\overset{H_m}{\longrightarrow}E_dH_m\text{ for some } d\geq 0$.
The conclusion follows from a result from~\cite{cp-19} which states that
for any $d \geq 3$ and $n \geq 1$ there exists a $\mathbb{Z}/p^{n 2^{d-3}}\mathbb{Z}$-equivariant map
\[
	E_d(\mathbb{Z}/p^{n 2^{d-3}}\mathbb{Z}) \longrightarrow E_3(\mathbb{Z}/p^n\mathbb{Z}).
\]
\end{proof}

\section{The noncommutative Borsuk--Ulam-type conjectures}\label{but}

In this section, we prove a noncommutative Borsuk--Ulam-type result for actions of compact quantum groups
having a classical subgroup whose induced action has finite local-triviality dimension. This result is an easy
consequence of the results presented in Sections~\ref{lt} and \ref{ltuniversal}.
Let us emphasize that the noncommtuative Borsuk--Ulam-type conjecture (see Conjecture~\ref{ncbdh} below) 
was one of the main reasons to introduce the local-triviality dimension,
because in the commutative setting, this conjecture is known exactly in the locally trivial case~\cite{cdt-16}.

The original Borsuk--Ulam antipodal theorem \cite[Satz II]{k-b33} can be equivalently formulated in the following way:
\begin{center}
{\em there is no map $S^{n+1}\to S^n$ intertwining the antipodal actions.}
\end{center}

This result was generalized to $q$-deformed spheres by Yamashita \cite[Corollary~15]{m-y13} 
and $\theta$-deformed spheres by Passer \cite[Corollary 4.6]{b-p16}. 

Matousek generalized the Borsuk--Ulam theorem as follows:
let $G$ be a finite group, then there is no $G$-equivariant map $E_{n}G\to E_{n-1}G$~\cite[Theorem~6.2.5]{Ma03}. 
If $G=\mathbb{Z}/2\mathbb{Z}$, we recover the usual Borsuk--Ulam theorem. 
Chirvasitu, D\k{a}browski and Hajac, based on the unpublished work of
Bestvina and Edwards, extended the aforementioned non-existence result to all compact Hausdorff groups~\cite{cdt-16}.
There is a natural question whether this result holds in the case of the $n$-fold equivariant noncommutative join
of a given compact quantum group.

\begin{conj}\label{ncbe}
Let $\mathbb{G}$ be a compact quantum group. Then there does not exists a $\mathbb{G}$-equivariant $*$-homomorphism
$C(E^\Delta_n\mathbb{G})\to C(E^\Delta_{n+1}\mathbb{G})$ and  
$\dim_{\rm LT}^\mathbb{G}(C(E^\Delta_n\mathbb{G}))=n$.
\end{conj}

An analogous conjecture for $n$-fold noncommutative free join of $\mathbb{G}$ would follow from Conjecture~\ref{ncbe}.  
Indeed, this fact can be deduced from the following result which is a direct consequence of the inequality 
\eqref{loceq} and Theorem~\ref{universal}. 

\begin{prop}
Let $\mathbb{G}$ be a compact quantum group and suppose that there exists
a single $\mathbb{G}$-action with the local-triviality dimension equal to $n$. 
Then there is no $\mathbb{G}$-equivariant $*$-homomorphism 
$C(E^{\mathrlap{+}\times}_n\mathbb{G})\to C(E^{\mathrlap{+}\times}_{n+1}\mathbb{G})$
and $\dim_{\rm LT}^\mathbb{G}(C(E^{\mathrlap{+}\times}_n\mathbb{G}))=n$.
\end{prop}
Since $\dim_{\rm LT}^G(C(E_nG))=n$ for any compact Hausdorff group $G$~\cite{cdt-16}, we arrive at:
\begin{cor}
Let $G$ be a compact Hausdorff group. 
Then there is no $G$-equivariant $*$-homomorphism 
$C(E^{\mathrlap{+}\times}_n G)\to C(E^{\mathrlap{+}\times}_{n+1}G)$ and $\dim_{\rm LT}^G(C(E^{\mathrlap{+}\times}_n G))=n$.
\end{cor}
Let us write the Borsuk--Ulam theorem for free noncommutative spheres as a~separate result.
\begin{cor}
There is no $*$-homomorphism $C(S^n_+)\to C(S^{n+1}_+)$ intertwining the antipodal actions
and $\dim_{\rm LT}^{\mathbb{Z}/2\mathbb{Z}}(C(S^n_+))=n$.
\end{cor}


Baum, D\k{a}browski and Hajac 
postulated the following noncommutative Borsuk--Ulam-type conjecture:
\begin{conj}[\cite{BDH15}]\label{ncbdh}
Let $A$ be a unital C*-algebra with a free action $\delta$ of a non-trivial compact quantum group $\mathbb{G}$.
There does not exist a $\mathbb{G}$-equivariant $*$-homomorphism $A\to A\overset{\delta}{\circledast}C(\mathbb{G})$.
\end{conj}

Using the notion of the local-triviality dimension, we can obtain a new noncommutative 
Borsuk--Ulam-type result. First, let us demonstrate that Conjecture~\ref{ncbdh}
for locally trivial actions on non-simple C*-algebras follows from Conjecture~\ref{ncbe}. 
Again, the notion of the $n$-universal compact quantum principal bundle is crucial for our arguments.

\begin{prop}\label{ltbdh}
Let $\delta:A\to A\otimes C(\mathbb{G})$ be an action of a compact quantum group~$\mathbb{G}$
on a unital C*-algebra $A$ such that $A$ admits a character and $\dim_{\rm LT}^\mathbb{G}(A)<\infty$.
Then Conjecture~\ref{ncbe} implies Conjecture~\ref{ncbdh}. 
\end{prop}
\begin{proof}
Suppose that we know that Conjecture~\ref{ncbdh} holds for $A$. 
Using universality of $C(E_n^{\mathrlap{+}\times}\mathbb{G})$ (Theorem~\ref{universal}), 
we obtain a chain of $\mathbb{G}$-$*$-homomorphisms
\[
C(E^{\mathrlap{+}\times}_n\mathbb{G})\to A\to A\overset{\delta}{\circledast}C(\mathbb{G})\to\ldots\to 
C(E_{n+1}^\Delta\mathbb{G}).
\]
Hence, to prove Conjecture~\ref{ncbdh} for locally trivial actions it is enough to prove that
$\dim_{\rm LT}^\mathbb{G}(C(E^\Delta_n\mathbb{G}))=n$, which is the exact formulation
of Conjecture~\ref{ncbe}.
\end{proof}
\begin{cor}\label{ltbdh2}
Let $\delta:A\to A\otimes C(G)$ be an action of a compact Hausdorff group~$G$
on a unital C*-algebra $A$ such that $A$ admits a character and $\dim_{\rm LT}^G(A)<\infty$.
There does not exist a~$G$-equivariant $*$-homomorphism $A\to A\overset{\delta}{\circledast}C(G)$.
\end{cor}
\begin{proof}
The result is a consequence of Proposition~\ref{ltbdh}, because Conjecture~\ref{ncbe} is true for
compact Hausdorff groups~\cite{cdt-16}.
\end{proof}

Finally, we can improve the result in Corollary~\ref{ltbdh2} to arbitrary unital C*-algebras equipped with an action 
of a compact quantum group
admitting a classical subgroup whose induced action has finite local-triviality dimension:

\begin{thm}
Let $\mathbb{G}$ be a compact quantum group, and let $A$ be unital C*-algebra  
equipped with a coaction $\delta:A\to A\otimes C(\mathbb{G})$. Then, if $\mathbb{G}$ admits
a non-trivial classical subgroup $H$ whose induced action $\alpha$ satisfies
$\dim_{\rm WLT}(\alpha)<\infty$, there is no $\mathbb{G}$-equivariant $*$-homomorphism 
$A\to A\overset{\delta}{\circledast} C(\mathbb{G})$.
\end{thm}
\begin{proof}
We closely follow the reasoning presented in \cite{Pa17}.
Suppose that there exists a $\mathbb{G}$-equivariant $*$-homomorphism
$A\to A\overset{\delta}{\circledast} C(\mathbb{G})$. Then by \cite{DHN17} there exists a $H$-equivariant
$*$-homomorphism $A\to A\circledast C(H)$. 
Evaluation at $e\in H$ in each point would produce a path of unital $*$-homomorphisms
on $A$ connecting a $H$-equivariant map to a one-dimensional representation. In what follows, we show that
such a path cannot exists.

Denote the aforementioned path of unital $*$-homomorphisms by $\psi_t:A\to A$, where $\psi_1$ is equivariant and $\psi_0:A\to 
\mathbb{C}$. Existence of $\psi_0$ implies that the $H$-invariant ideal $I=\langle ab-ba : a,b\in A\rangle$ is proper and we
can consider the abelianization $A_c:=A/I$. By the Gelfand-Naimark theorem, 
$A_c=C(X)$ for some compact Hausdorff space~$X$. Since
the action of $H$ on $A$ has finite weak local-triviality dimension, by inequality~(\ref{quotient}), the induced action on $A_c$ has finite 
weak local-triviality dimension as well. In terms of spaces, this means that the principal $H$-bundle $X\to X/H$ is locally trivial.

Now for every $t\in[0,1]$, $\psi_t$ induces a map $\widetilde{\psi}_t:A_c\to A_c$. Such a path is dual to the homotopy of maps on 
the space $X$ connecting an $H$-equivariant map with a constant map. This implies equivariant contractibility of $X$, which is
equivalent with the existence of an $H$-map $X\ast H\to X$. However, this map cannot exists by~\cite{cdt-16}.
\end{proof}

\section*{Acknowledgement}\noindent
Authors are grateful to Ludwik D\k{a}browski
for valuable suggestions regarding the exposition of the paper and to G\'abor Szab\'o for discussions concerning c.p.c. order
zero maps. M.T. would also like to thank Ben Passer and Alex Chirvasitu 
for many conversations about applications of the presented results to the Borsuk--Ulam-type
conjectures.
This work is part of the project Quantum Dynamics supported by
EU-grant RISE 691246 and Polish Government grant 317281.
M.T. was
partially supported by the project Diamentowy Grant No.~DI2015 006945
financed by the Polish Ministry of Science and Higher Education.

\end{document}